\documentclass[12pt, a4paper]{amsart}
\usepackage[margin=1.2in]{geometry}
\usepackage[normalem]{ulem} 
\usepackage[Verbose]{parallel}
\usepackage{amsmath}
\usepackage{amsthm}
\usepackage{color}
\usepackage{amsfonts}
\usepackage{hhline}
\usepackage{amssymb}
\usepackage{amscd}
\usepackage{amsmath}
\usepackage{tikz}
\usetikzlibrary {shapes} 
\usetikzlibrary{arrows}
\usetikzlibrary{plothandlers}
\usetikzlibrary{calc}

\usepackage{booktabs}
\usepackage{array}
\usepackage{multirow}

\numberwithin{equation}{section}
\newtheorem{theorem}{Theorem}[section]
\newtheorem*{thkp}{Theorem (Kostant, Panyushev)}
\newtheorem*{theore}{Theorem}

\newtheorem{lemma}[theorem]{Lemma}
\newtheorem{prop}[theorem]{Proposition}

\newtheorem{defi}{Definition}[section]

\theoremstyle{definition}
\newtheorem{example}{Example}[section]
\newtheorem{rem}{Remark}[section]

\newcommand\I{\Cal I}

\newcommand\lu{{\overline{\ell}}}

\newcommand\ov{\overline}

\renewcommand\({\left(}
\renewcommand\){\right)}

\newcommand\be{\beta}

\newcommand\g{\mathfrak g}

\newcommand\h{\mathfrak h}

\newcommand\ha{\widehat{\mathfrak h}}

\newcommand\bb{\mathfrak b}
\newcommand\D{\Delta}
\renewcommand\l{\lambda}
\newcommand\Dp{\Delta^+}

\newcommand\Da{\widehat\Delta}
\newcommand\Pia{{\widehat\Pi}}
\newcommand\Dap{\widehat\Delta^+}
\newcommand\Wa{\widehat{W}}
\renewcommand\d{\delta}

\renewcommand\a{\alpha}

\renewcommand\b{\bb}

\renewcommand\k{\mathfrak k}

\renewcommand\th{\theta}

\renewcommand\i{{\mathfrak  i}}

\newcommand\nat{\mathbb N}
\newcommand\ganz{\mathbb Z}

\newcommand\s{\sigma}

\newcommand\x{{\bf x}}

\newcommand\real{\mathbb R}

\newcommand\C{\mathbb C}
\newcommand\R{\mathbb R}
\newcommand\si{\sigma}

\newcommand\G{\Gamma}

\renewcommand\ha{\widehat{\mathfrak h}}

\renewcommand\I{\mathcal I}

\newcommand{\Wab}{\mathcal W_\s^{ab}}
\renewcommand{\S}{\Sigma}
\renewcommand{\G}{\Gamma}
\begin{document}
\title[Borel stable abelian subalgebras]{On the structure of Borel stable abelian subalgebras in infinitesimal symmetric spaces}
\author[Cellini]{Paola Cellini}
\author[M\"oseneder Frajria]{Pierluigi M\"oseneder Frajria}
\author[Papi]{Paolo Papi}
\author[Pasquali]{Marco Pasquali}
\begin{abstract} Let $\g=\g^{\bar 0}\oplus \g^{\bar 1}$ be a $\ganz_2$-graded Lie algebra. We study the posets of abelian subalgebras of $\g^{\bar 1}$ which are stable 
w.r.t. a Borel subalgebra of $\g^{\bar 0}$. In particular, we find out a natural parametrization of maximal elements and dimension formulas for them. 
We recover as special cases several results of Kostant, Panyushev,
 Suter.
\end{abstract}
\maketitle
\tableofcontents
\section{Introduction}
Let $\g$ be a finite dimensional complex  semisimple Lie algebra.
Let $\si$ be an  involution of $\g$ and $\g=\g^{\bar 0}\oplus \g^{\bar 1}$ be  the corresponding eigenspace decomposition.  Fix a Borel subalgebra $\b^{\bar 0}$ of the reductive Lie algebra $\g^{\bar 0}$. In this paper we deal with the following problem: {\it parametrize the maximal abelian $\b^{\bar 0}$-stable subalgebras of $\g^{\bar 1}$ and find  formulas for their dimension.}\par
This kind of problem has ancient roots. A  prototypical version of it is Schur's theorem \cite{Schur}, stating that there  exist
at most $\lfloor\frac{N^2}{4}\rfloor+1$ linearly independent commuting matrices in $gl(N)$. To make a long story short, developments related to Schur's result (whose proof had been simplified by Jacobson \cite{Jacobson} in the 50's) can be summed up as follows. 
\vskip7pt
\begin{Parallel}{0.48\textwidth}{0.48\textwidth}
  \ParallelLText{1945: Malcev \cite{Mal} found the maximal dimension of an abelian subalgebra of any simple $\g$.}
\ParallelRText{}
\end{Parallel}
\begin{Parallel}{0.48\textwidth}{0.48\textwidth}
 \ParallelLText{
 1965: Kostant  \cite{K1}
found a connection between the 
  eigenvalues  of a Casimir of $\g$ 
  and  the commutative subalgebras of $\g$.
  }
  \ParallelRText{2001: Panyushev \cite{Pan} ge\-ne\-ra\-lized Kostant's 
 results to the gra\-ded set\-ting.
 \phantom{phantom}  \phantom{phantom}}
\end{Parallel}
\begin{Parallel}{0.48\textwidth}{0.48\textwidth}
   \ParallelLText{2000: Peterson's Abelian ideals theorem (cf. \cite{KoIMRN}):  the abelian ideals of a Borel subalgebra of $\g$ are $2^{rk(\g)}$.}
\ParallelRText{2004: Cellini-M\"oseneder-Papi  found 
a uniform  enumeration of $\b^{\bar 0}$-stable abelian subalgebras of $\g^{\bar 1}$ (cf. \cite{IMRN}).}
\end{Parallel}
\begin{Parallel}{0.48\textwidth}{0.48\textwidth}
   \ParallelLText{2003: Panuyshev  \cite{Panadv} found a natural bijection
   between maximal abelian ideals of a Borel subalgebra of a simple Lie algebra  and long simple roots.}
    \ParallelRText{}
   \ParallelPar
   \ParallelLText{2004: Suter \cite{Suter} gave a conceptual explanation 
   of Malcev's result, providing a uniform formula for the dimension of maximal abelian ideals of a simple Lie algebra.}
  \ParallelRText{}  
\end{Parallel}
\vskip5pt
In these terms, solving our initial problems means filling in the missing slots in the right column. Indeed, what  links all these problems is their interpretation in terms of 
$\widehat{\mathfrak u}$-cohomology,  $\widehat{\mathfrak u}$ being the nilpotent radical of the parabolic subalgebra $(\C[t]\otimes \g)\cap \widehat L(\g,\si)$ in the affine
Kac-Moody algebra $\widehat L(\g,\si)$. This remark is at the basis of Kostant's paper 
 \cite{Koinv}, and it is generalized to the graded setting in
 \cite{MP}. In the latter paper  it is shown that  combining Garland-Lepowsky theorem on
$\widehat{\mathfrak u}$-cohomology with  the relationships between 
 the Laplacian associated to the standard Eilenberg-Chevalley boundary and the Casimir elements of $\widehat L(\g,\si)$ and $\g^{\bar0}$, it is possible to prove the following results, which motivate and give applications to our initial problem.\par
Given a commutative subalgebra $\mathfrak a$ of $\g^{\bar 1}$ with basis $v_1,\ldots,v_k$, consider the vector $v_{\mathfrak a}=v_1\wedge\cdots\wedge v_k\in \Lambda^k\g^{\bar 1}$ and let $A_k$ be the span of the $v_{\mathfrak a}$'s when 
$\mathfrak a$ ranges over the $k$-dimensional commutative subalgebras  of $\g^{\bar 1}$. Let finally $m_k$ be the maximal eigenvalue of the Casimir element of $\g^{\bar 0}$ w.r.t. the Killing form of $\g$ on 
$\Lambda^k\g^{\bar 1}$ and $M_k$ the eigenspace of eigenvalue $k/2$.
\begin{thkp}\ \begin{enumerate}
\item $m_k\leq k/2$;
\item $m_k= k/2$ if and only if $A_k\ne \emptyset$. In such a case $A_k=M_k$;
\item $A=\sum_kA_k$ is a multiplicity free $\g^{\bar 0}$-module whose irreducible pieces are indexed by the $\b^{\bar 0}$-stable abelian subalgebras of $\g^{\bar 1}$.
\end{enumerate}
\end{thkp}
Another result which is naturally explained by the cohomological approach is Peterson's theorem quoted above. This theorem admits an interpretation in terms of the geometry of alcoves, which we presently explain. Let $\b$ be a Borel subalgebra of $\g$, which we temporarily assume to be simple. An abelian ideal $\i$ of $\b$, being stable w.r.t. the Cartan component of $\b$,  is a sum of root subspaces relative to a dual order ideal
$A$
of positive roots of $\g$. Peterson's trick consists in considering the set of positive {\it affine} roots $-A+\d$, $\d$ being the fundamental imaginary root of $\widehat L(\g,\si)$. It's easy to check that this set is biconvex, hence is a set of generalized inversions of an element $w\in\Wa$, the Weyl group of  $\widehat L(\g,\si)$ (see Subsection \ref{is}).  
Peterson calls {\it minuscule} an element $w\in\Wa$  associated,  according to the above procedure, to an abelian ideal. It is shown in  \cite{CP1} that $w$ is minuscule   if and only if $wC_1\subset 2C_1$, $C_1$ being the fundamental alcove, i.e. a fundamental domain for the affine action of $\Wa$ on  $(\h_0)_\R^*$. This fact explains the enumerative result.  It  can be rephrased by saying that there exists a suitable simplex in $(\h_0)_\R^*$ paved by the abelian ideals. 
The graded generalization found in \cite{IMRN}, though much more complicated, is in the same spirit: the $	\b^{\bar 0}$-stable abelian subalgebras of $\g^{\bar 1}$ are indexed  by alcoves in a polytope $D_\si$, of which  explicit equations are provided. This result will be recalled and refined in Section \ref{bs}, and is the starting point for our investigation  of maximal $\b^{\bar 0}$-stable abelian subalgebras of $\g^{\bar 1}$. Let $\Wab$ be the subset of $\Wa$ formed by the elements indexing the alcoves of $D_\si$.  We locate 
a special subset $\mathcal M_\si$ (see \eqref{bw}) of bounding walls, with the property that if $w$ is maximal (i.e., the corresponding $\b^{\bar 0}$-stable abelian subalgebra is), then $w(C_1)$ has a face on $\mathcal M_\si$. We are therefore reduced to study the posets
$$\mathcal I_{\a,\mu}=\{w\in\Wab\mid w(\a)=\mu\},$$
where $\a$ is a simple root of $\widehat L(\g,\si)$ and $\mu\in \mathcal M_\si$.
This is done according to the following steps.
\begin{itemize}
\item We provide a criterion for  $\mathcal I_{\a,\mu}$ to be non empty.
Moreover we show that $\mathcal I_{\a,\mu}$, if non empty,  has minimum
(Theorem \ref{min}).
\item We determine the poset structure of $\mathcal I_{\a,\mu}$, by relating it to a 
quotient of the subgroup $\Wa_\a$ of $\Wa$ generated by the simple reflections orthogonal to $\a$ by a reflection subgroup $\Wa'_\a$ (Theorem \ref{slamu}).
\item We look at intersections among the posets $\mathcal I_{\a,\mu}$, and we find necessary and sufficient conditions in order  that the intersection of two such posets is nonvoid.
\item We study maximal elements in $\mathcal I_{\a,\mu}$. We show that when 
$\Wa'_\a$ is not standard parabolic, maximal elements appear in pairs  of $\mathcal I_{\a,\mu}$'s: if $w$ is maximal in 
$\mathcal I_{\a,\mu}$, then there exist a unique simple root $\beta$ and a unique wall $\mu'\in	\mathcal M_\s$ such that $w$  is also maximal in $\mathcal I_{\beta,\mu'}$ (Lemma \ref{coppie}).
\item We determine which maximal elements in $\mathcal I_{\a,\mu}$ are indeed maximal in $\Wab$ (Propositions \ref{mmu}, \ref{mmassimo}).
\end{itemize}\par
We finally provide a complete parametrization of maximal abelian $\b^{\bar 0}$-stable subalgebras (Theorem \ref{spazioparametri}) and uniform formulas for their dimension 
(Corollary \ref{dim}). Our results specialize nicely to Panyushev's and Suter's theorems
quoted above (see Remark \ref{PS}). But it is worthwhile to note that new pheno\-mena appear, like the presence of maximal subalgebras indexed by certain pairs of simple roots lying in different components of $\D_0$. To illustrate this fact, we state here our result in the special case when $\g^{\bar 0}$ is semisimple and $\s$ is of  inner type. In this case, by Kac's theory, we can choose a set of simple roots $\Pi$ for $\g$ in such a way that there is 
a  unique simple root $\tilde\a\in\Pi$ such that $	\s(x)=-x$ for $x\in \g_{\tilde\a}$ and 	$\s(x)=x$  for $x\in \g_{\beta}$ with $\beta\in \Pi\setminus\{\tilde\a\}$.
Let $\Dp_0=\coprod_{i=1}^r\Dp(\S_i)$ be the decomposition of the (positive) root system of $\g^{\bar 0}$ into irreducible subsystems. Set $\mu_i=\d-\theta_{\S_i},\,1\leq i\leq r$, $\theta_{\S_i}$ being the highest root of $\D(\S_i)$.
\begin{theore} In the above setting, the  maximal $\b^{\bar 0}$-stable abelian subalgebras of $\g^{\bar 1}$ are:
\begin{itemize}\item $\max \mathcal I_{\a,\mu_i},\,1\leq i\leq r$, $\a$ being a long simple root in $\G(\S_i)$ (see Def. \ref{Asigma}) if $\th_{\S_i}$ is long or 
in $\S_i$ if $\th_{\S_i}$ is short;
\item $\max \mathcal I_{\tilde\a,\tilde\a+\d}$;
\item $\max \left(\mathcal I_{\a,\mu_i}\cap \mathcal I_{\be,\mu_j}\right),\,1\leq i<j\leq r, \a\in\S_j,\be\in\S_i$ being  long simple roots.
\end{itemize}
The dimensions of these maximal subalgebras are given by formulas \eqref{dim1}, \eqref{dim1bis}, \eqref{dim2}.
\end{theore}
\par

\vskip5pt
\section{Setup}
\subsection{Twisted loop algebra and automorphisms}\label{ta}
Let $\g, \si$ be as in the Introduction.  We assume that $\si$ is indecomposable,  i.e. $\g$ has no nontrivial $\s$-invariant ideals. Let
$(\cdot,\cdot)$ be the Killing form of $\g$.  For $j\in\ganz$ set $\bar j=j+2\ganz$, and let $\g^{\bar j}=\{X\in \g\mid \s(X)=(-1)^jX\}$, so that we have  $\g=\g^{\bar 0}\oplus \g^{\bar1}$. 
We let $\widehat L(\g,\s)$ be the affine Kac-Moody Lie algebra associated to $\s$ in \cite[Section 8.2]{Kac}.
Let  $\h_0$ be a Cartan subalgebra of $\g^{\bar 0}$. As shown in \cite[Chapter 8]{Kac},  $\h_0$ contains a regular element $h_{reg}$ of $\g$. In particular the centralizer $Cent(\h_0)$ of $\h_0$ in $\g$ is a Cartan subalgebra of $\g$ and $h_{reg}$ defines a set of positive roots in the set of roots of $(\g, Cent(\h_0))$ and a set $\Dp_0$ of positive roots in the set $\D_0$ of roots for $(\g^{\bar 0},\h_0)$. Since $\s$ fixes $h_{reg}$,  we see that the action of $\s$ on the positive roots defines, once  Chevalley generators are fixed, a dia\-gram automorphism $\eta$ of $\g$ that, clearly, fixes $\h_0$. Set, using the notation of \cite{Kac}, $\ha=\h_0\oplus\C K\oplus\C d$. Recall that $d$ is the element of $\widehat L(\g,\s)$ acting on $\widehat L(\g,\s)\cap(\C[t,t^{-1}]\otimes \g)$ as $t\frac{d}{dt}$, while $K$ is a central element.
Define 
$\d'\in\ha^*$  by setting $\d'(d)=1$ and $\d'(\h_0)=\d'(K)=0$ and let $\l\mapsto \ov\l$ be the restriction map $\ha\to\h_0$. 
There is a unique extension, still denoted by $(\cdot,\cdot)$,  of the 
 Killing form of $\g$ to a nondegenerate symmetric bilinear invariant form on 
 $\widehat L(\g,\si)$.  Let $\nu:\ha\to\ha^*$ be 
the isomorphism induced by the form $(\cdot,\cdot)$, and denote again by $(\cdot,\cdot)$
the form induced on $\ha^*$. One has  $(\d',\d')=(\d',\h_0^*)=0$.

We let $\Da$ be the set of $\ha$-roots of $\widehat L(\g,\s)$. We can choose as set of  positive roots $\Dap=\Dp_{0}\cup\{\a\in\Da\mid \a(d)>0\}$. We let $\Pia=\{\a_0,\dots,\a_n\}$ be the corresponding set of simple roots. It is known that $n$ is the rank of $\g^{\bar0}$.  Recall that any $\widehat L(\g,\s)$ is a Kac-Moody Lie algebra $\g(A)$ defined by generator and relations starting from a generalized Cartan matrix $A$ of affine type. These matrices are
classified by means of Dynkin diagrams listed in \cite{Kac}.

Following \cite[Chapter 8]{Kac}, we can assume that $\s$ is the automorphism of type $(\eta;s_0,\dots, s_n)$, where $\eta$ is the automorphism of the diagram defined above. Note that, since $\s$ is an involution, $\eta^2=Id$. We do not assume here that $\g$ is simple, but, as explained in \cite{MMJ}, most arguments given in \cite{Kac} can be safely extended to the  setting where $\g$ is semisimple but not simple. This latter case, i.e. $\g=\k\oplus\k,\,\k$ a simple Lie algebra, $\s$ the flip, will be referred to as the {\it adjoint case}.  Recall that, if $a_0,\dots,a_n$ are the labels of the Dynkin diagram of $\widehat L(\g,\s)$ and $k$ is the order of $\eta$, then 
$k(\sum_{i=0}^ns_ia_i)=2$. Recall also that $s_0,\dots,s_n$ are relatively prime so we must have that $s_i\in\{0,1\}$ and $s_i= 0$ for all but at most two indices.
The case in which we have two indices equal to $1$ will be referred to as the {\it hermitian case} (indeed $\g/\g^{\bar 0}$ is an infinitesimal hermitian symmetric space).
Since $\s$ is the automorphism of type $(\eta;s_0,\dots, s_n)$, we can write $\a_i=s_i\d'+\ov{\a_i}$ and the set $\Pi_0=\{\a_i\mid s_i=0\}$ is the set of simple roots for $\g^{\bar 0}$ corresponding to $\Dp_0$.
Set also $\Pi_1=\Pia\setminus \Pi_0$.

Introduce $\d=\sum_{i=0}^na_i\a_i$ and note that $\d=(\sum_{i=0}^na_i s_i)\d'=\frac{2}{k}\d'$.
Set also $\a_i^\vee=\frac{2}{(\a_i,\a_i)}\nu^{-1}(\a_i)$ and let $\{a^\vee_0,\dots,a^\vee_n\}$ be the labels of the dual Dynkin diagram of $\widehat L(\g,\s)$. 

We assume that $K$ is the canonical central element \cite [6.2]{Kac},  $K=\sum_{i=0}^n a_i^\vee \a_i^\vee$. If we number the Dynkin diagrams as in \cite[Tables Aff1, Aff2, Aff 3]{Kac} then, by Sections  6.1, 6.2, 6.4 of \cite{Kac}, 
\begin{equation}\label{C}
K=\frac{2a_0}{\Vert \d-a_0\a_0\Vert}\nu^{-1}{(\d)}.
\end{equation}

Set finally  ${\bf g}=\sum_{i=0}^n a_i^\vee$. This number is called the dual Coxeter number of $\widehat L(\g,\s)$.

We let $\Wa$ be the Weyl group of $\widehat L(\g,\s)$. Set $(\h_0)_\R=\oplus_{\a\in\Pi}\R\a^\vee$ and $\ha_\R=\R d\oplus \R K\oplus(\h_0)_\R$. Set
\begin{equation}\label{caf}
C_1=\{h\in(\h_0)_\R\mid \ov\a_i(h)\geq -s_i,\,i=0,\ldots,n\}
\end{equation}  
be the fundamental alcove of $\Wa$. \subsection{Combinatorics of inversion sets}\label{is}  For $w\in\Wa$, we set $$N(w)=\{\a\in\Dap\mid w^{-1}(\a)\in-\Dap\}.$$
 If $\a$ is a real root in $\Dap$, we let $s_\a$ denote the reflection in $\a$. If $\a_i$ is a simple root we set $s_i=s_{\a_i}$.\par
The following facts are well-known. More details and references can be found in \cite{CP3}. We will often  use these properties  in the rest of the paper without further notice.
\begin{enumerate}
\item $N(w_1)=N(w_2)\implies w_1=w_2$. 
\item if $w=s_{i_1}\cdots s_{i_k}$ is a reduced expression for $w$, then $$N(w)=\{\a_{i_1}, s_{i_1}(\a_{i_2}),\ldots,s_{i_1}\cdots s_{i_{k-1}}(\a_{i_k})\};$$
if moreover $\beta_h=s_{i_1}\cdots s_{i_{h-1}}(\a_{i_h}),\,1\leq h\leq k$, then 
\begin{equation}\label{NW}
w=s_{\beta_k}s_{\beta_{k-1}}\cdots s_{\beta_1}.
\end{equation}
\item $N(w)$ is biconvex, i.e. both $N(w)$ and $\Dap\setminus N(w)$ are closed under root addition. Conversely, if $\Dap$ has no irreducible components of type $A_1^{(1)}$ and $L$ is a finite subset
of real roots which is biconvex, then there exists $w\in\Wa$ such that $L=N(w)$.
\item Denote by $\le$  the weak left Bruhat order: $w_1\le w_2$ if there exists a reduced expression for $w_1$ which is an initial segment of a reduced expression for $w_2$). Then
 $$w_1<w_2 \iff N(w_1)\subset N(w_2).$$
 \item Set $N^{\pm}(w)=N(w)\cup-N(w)$. Then $N^\pm(w_1w_2)=N^\pm(w_1)\dot{+}w_1(N^\pm(w_2))$, where $\dot{+}$ denotes the symmetric difference. In particular, the following properties are equivalent:
 \begin{enumerate}
 \item $N(w_1w_2)=N(w_1)\cup w_1(N(w_2))$;
 \item $\ell(w_1w_2)=\ell(w_1)+\ell(w_2)$;
 \item $w_1(N(w_2))\subset\Dap$.
 \end{enumerate}
\end{enumerate}
 We also introduce the sets of left and right descents for $w\in\Wa$:
\begin{align*}
L(w)&=\{\a\in\Pia\mid \ell(s_\a w)<\ell(w)\},\\
R(w)&=\{\a\in\Pia\mid \ell(ws_\a)<\ell(w)\}.
\end{align*}
We have that $L(w)=\Pia\cap N(w),\,R(w)=\Pia\cap N(w^{-1})$.
\subsection{Conventions on root systems}
\subsubsection{} We  number affine Dynkin diagrams as in \cite[Tables Aff1 and Aff2]{Kac}.
\subsubsection{} If $v\in \ha^*$, we set  $v^\perp=\{x\in \ha^*\mid (x,v)=0\}$.
\subsubsection{} If $S\subseteq\Pia$, we denote  by $\D(S)$ (resp. $\Dp(S)$) the root system generated by $S$ (resp. the set of positive roots corresponding to $S$). If 
$A\subseteq \Dap$  we   denote by $W(A)$ the Weyl group generated (inside $\Wa$) by the reflections in the elements of  
$A$.\par
We often identify subsets of the set of simple roots with their Dynkin diagram.\par
\subsubsection{} If $R$ is a finite or affine root system and $\Pi_R$ is a basis of simple roots, we write the expansion of a root $\gamma\in R$ w.r.t. $\Pi_R$ as
\begin{equation}\label{c}
\gamma=\sum_{\a\in\Pi_R}c_\a(\gamma)\gamma.
\end{equation}
We also set, for $\a\in R$,
$$supp(\a)=\{\be\in \Pi_R\mid c_\be(\a)\ne 0\}.$$
\subsubsection{} If $R$ is a finite irreducible root system and $\Pi$ is a set of simple roots for $R$, we denote by $\theta_R$ (or by $\theta_\Pi$) its highest root.  Recall that the highest root and the highest short root are the only dominant weights belonging to $R^+$. We will use this remark in the following form:
$$\a\in R^+,\,\a \text{ long },\,(\a,\beta)\geq 0\,\forall\,\beta\in R^+\implies \a=\theta_R.$$
\par
 \subsubsection{} We recall the definition of dual Coxeter number $g_R$ of a finite irreducible root system $R$. Write
$\theta_R^\vee=\sum_{\a\in\Pi_R}c_{\a^\vee}(\theta^\vee)\a^\vee$ and set
\begin{equation}\label{dcn}
g_R=1+\sum_{\a\in\Pi_R}c_{\a^\vee}(\theta^\vee).
\end{equation}

\subsection{Reflection subgroups and coset representatives} Let $G$ be a finite or affine reflection group and let $\ell$ be the length function with respect to a fixed set of Coxeter generators $S$. Let $R$ be the set of roots of $G$ in the geometric representation, $\Pi_R$ a system of simple roots for $R$, and $R^+$ the corresponding set of positive roots. 
Let $G'$ be a  subgroup of $G$ generated by reflections, and $R'$ be the set of roots $\a\in R$ such that $s_\a\in G'$, which is easily shown to be a root system.
By \cite{Deo}, 
 $$\Pi_{R'}=\{\a\in R^+\mid N(s_\a)\cap R'=\{\a\}\}$$
 is a set of simple roots for $R'$, whose associated set of positive roots is
$R'^+=R'\cap R^+$. \par
Given $g\in G$, we say that an element $w\in G'g$ is a minimal right  coset representative if $\ell(w)$ is minimal among the lengths of elements of $G'g$. 
It follows from \cite{Deo} by a standard argument that  a coset $G'g$ has  a unique minimal right coset representative $w$ and this element is characterized by the following property:
\begin{equation}\label{mcr}w^{-1}(\a)\in R^+\ \text{ for all $\a\in R'^+$}.\end{equation}

We will always choose as a coset representative for $G'g$  the minimal right  coset representative and (with a slight abuse of notation) we denote  by  $G'\backslash G$ the set of all minimal right coset representatives. Thus the restriction of the weak order of $G$ on 
$G'\backslash G$ induces a partial ordering on
$G'\backslash G$. When saying {\it the poset}  $G'\backslash G$, we
shall always refer to this ordering.
\par

\subsubsection{}\label{orbit}
If $\a\in R$ and $G'$ is the stabilizer of $\a$ in $G$, then, for each $g\in G$,  the minimal length  representative of $G'g$ is the unique minimal length element that maps $g^{-1}\a$ to $\a$. By formula \ref{mcr}, this element is characterized by the 
property
 \begin{equation}\label{mcr+}w^{-1}(\beta)\in R^+\ \text{ for all $\beta\in R^+$   orthogonal to $\a$.} \end{equation}

\smallskip
\subsubsection{}\label{posetparabolic} A reflection subgroup $G'$ of $G$ is   standard parabolic when $\Pi_{R'}\subseteq \Pi_R$.  In this case, if  $g\in G$ and $w$ is the minimal right coset representative of $G'g$, then $g=g'w$ with $g'\in G'$ and $\ell(g)=\ell(g')+\ell(w)$. In particular $N(g)\cap
R'=N(g')$. Moreover, it is well known that  $g$ itself is the minimal representative of $G'g$ if and only if $L(g)\subseteq \Pi_R\setminus \Pi_{R'}$. Therefore $G'\backslash G=\{w\in G\mid L(w)\subseteq  \Pi_R\setminus \Pi_{R'}\}$. If $G$ is finite, the poset $G'\backslash G$ has a unique minimal and a unique maximal
element. The identity of $G$  clearly corresponds to the minimum of
$G'\backslash G$. If $w_0$ is the longest element of $G$ and $w_0'$
is the longest element of $G'$, then we have that $N(w_0'w_0)=
R\setminus R'$. If $w\in G'\backslash G$, then $N(w)\subseteq R\setminus R'$; therefore  $w_0'w_0$
is the unique maximal element of  $G'\backslash G$.
Note that 
\begin{equation}\label{poset}
\ell(w_0'w_0)=|\Dp(R)|-|\Dp(R')|.
\end{equation}
\subsection{Special elements in finite Weyl groups}
We sum up in the following statement the content of Propositions 7.1 and 7.2 from \cite{CP3}. Attributions of the individual  results are done
there.  The properties below will be used many times in the sequel.

\begin{prop}\label{theta} Let $R$ be a finite irreducible root system, $W_R$ its Weyl group. Fix  a positive system $R^+$ and let $\Pi_R,\theta_R$ be the corresponding set of simple root and highest root, respectively.
\begin{enumerate} \item\label{1} For any long root $\a$ there exists a unique element $y_\a\in W_R$ of minimal  length such that $y(\a)=\theta_R$.
\item\label{2} $L({y_\a})\subset \{\beta\in\Pi_R\mid (\beta,\theta_R)\ne 0\}$.
\item\label{3} If conversely $v\in W_R$ is such that $v(\a)=\theta_R$ and $L(v)\subset \{\beta\in\Pi_R\mid (\beta,\theta_R)\ne 0\}$, then $v=y_\a$.
\item\label{4} If $\a\in \Pi_R$, then $\ell(y_\a)=g_R-2$, 
$g_R$ being the dual Coxeter number of $R$.
\item\label{5} If $\a\in\Pi_R$, and $\beta_1+\beta_2=\theta_R,\,\beta_1,\beta_2\in R^+$, then exactly one element among $\beta_1,\beta_2$ belongs to  $N(y_\a)$, and any element of  $N(y_\a)$ arises in this way.
\item\label{6} Conversely, if $y\in W_R$ is such that for any pair $\beta_1,\beta_2\in R^+$
such that $\be_1+\be_2=\theta_R$ exactly one of $\be_1,\be_2$ belongs to $N(y)$ and 
$\theta_R\notin N(y)$, then there exists a long simple root $\beta$ such that 
$y(\beta)=\theta_R$.
\item\label{7} $N(y_\a^{-1})=\{\beta\in R^+\mid (\beta,\a^\vee)=-1\}$.
\item\label{8} $\gamma\in R^+,\,(\gamma,\theta_R)=0\implies \gamma\notin N(y_\a)$.
\end{enumerate}
\end{prop}

 \section{Borel  stable abelian subalgebras and affine Weyl groups}\label{bs}
  \vskip5pt
Recall that $\Pi_0$ denotes  the set of simple roots of $\g^{\bar0}$ corresponding to $\Dp_0$.
 In general $\Pi_0$ is disconnected and we write $\S|\Pi_0$ to mean that $\S$ is a connected component of $\Pi_0$.   Clearly,  the Weyl group $W_0$ of $\g^{\bar 0}$ is the direct product of the $W(\S),\,\S|\Pi_0$.
If $\theta_\S$ is the highest root of $\D(\S)$, set
 \begin{align*}
&\Da_0=\{\a+\ganz k\d\mid \a\in\D_0\}\cup\pm\nat k\d,\\
 &\Pia_0=\Pi_0\cup\{k\d-\theta_\S\mid \S|\Pi_0\},\\ 
 &\Dap_0=\Dp_0\cup\{\a\in\Da_0\mid \a(d)>0\}.\end{align*}
 Denote by $\Wa_0$ the  Weyl group of $\Da_0$.
 Let  $\Da_{re}=\widehat W\Pia$ be the set of real roots of $\widehat L(\g,\s)$. If $\l\in\h_0^*$, then we let $\g_\l\subset \g$ be the corresponding weight space.
We say that a real root $\a$ is noncompact if $\g_{\ov\a}\subset \g^{\bar 1}$, compact if $\g_{\ov\a}\subset \g^{\bar 0}$, and complex if it is neither compact nor noncompact. Note that, by the very definition of $\widehat L(\g,\s)$, if $\a\in\Da_{re}$, then $k\d+\a\in\Da$, while, if $k=2$, $\d+\a\in\Da$ if and only if $\a$ is complex. 
Clearly, if $\eta=Id$, then any real root is either compact or noncompact. It is shown in \cite{CKMP} that, if $\g$ is simple and $\eta\ne Id$, then a real root $\a$ is either compact or noncompact if and only if  $\a$ is a long root (i.e.,  $\Vert\a\Vert$ is largest among the possible root lengths). If $\g$ is not simple, since $\s$ is indecomposable, all the real roots are complex.

If $\a\in\Da$, set (cf. \eqref{c})
 $$
 ht_\s(\a)=\sum_{i=0}^n s_ic_{\a_i}(\a)
 $$
 and, for  $i\in\ganz$,   
$$\Da_i=\{\a\in\Da\mid ht_\s(\a)=i\}.$$

\begin{rem}\label{hts} 
Since $\a_i=s_i\delta'+\ov\a_i$ (Section \ref{ta}), for any $\a\in \Da$, we have that  $\a=ht_\s(\a)\d'+\ov\a$. In particular,  since $k\delta=2\delta'$, $ht_\s(k\delta)=2$.
By definition, the roots $\theta_\S$, $\S|\Pi_0$, are the maximal roots having  $\s$-height equal to $0$, with respect to the usual order $\le$ on roots: $\a\le \be$ if and only if $\be-\a$ is a sum of positive roots or zero. 
It follows that the roots $k\d-\theta_\S$ are the minimal roots  having  $\s$-height equal to $2$. More generally,  if $s\in\ganz$, $\{sk\d-\theta_\S\mid \S|\Pi_0\}$ is the set of  minimal roots in  $\Da_{2s}$. Similarly, 
$\Pi_1+sk\d$  is the set of  minimal roots in $\Da_{2s+1}$. 
\end{rem}

\begin{defi} An element $w\in \Wa$ is called $\s$-minuscule if $N(w)\subset \Da_1$. 
We denote by $\Wab$ the set of $\s$-minuscule elements of $\Wa$. 
\end{defi}

We regard $\Wab$ as a poset under the weak Bruhat order.

\begin{rem} Note that in the adjoint case $\g=\k\oplus \k$, $\k$ simple, $w$ is $\si$-minuscule if and
only if $N(w)\subset -\Dp_\k+\d$, $\Dp_\k$ being the  set of positive roots of $\k$. So we recover Peterson's notion of minuscule elements quoted in the Introduction.
\end{rem}

\begin{rem} It will be useful, from a notational point of view, to introduce the following generalization of the $\s$-height. Given $A\subseteq\Pia$ and $\gamma\in\Da$, set
$$ht_A(\gamma)=\sum_{\a\in A}c_{\a}(\gamma).$$
In particular, the $\s$-height equals $ht_{\Pi_1}$ and the usual height equals $ht_{\Pia}$. In these two cases we will  keep using $ht_\s,\,ht$.
\end{rem}

\vskip5pt
 Let $a$ be the squared length of a long root in $\Dap$. Define
\begin{equation}\label{pi0*}
\Pia_0^*=
\Pi_0\cup
\left\{ k\d-\theta_\S\mid a\leq2\Vert\theta_\S\Vert^2\right\},
\end{equation}
\begin{equation}\label{phis}\Phi_\s=
\Pia^*_0\cup\{\a+k\d\mid \a\in\Pi_1, \a \text{ {long and noncomplex}}\}
\end{equation}
\vskip5pt

\begin{rem}\label{phisigma} \item{(1)}
It is immediate to see that  $\Pia_0^*=\Pia_0$, unless $\widehat L(\g,\s)$ is of type $G_2^{(1)}$ or $A_2^{(2)}$.  Indeed, in the latter cases there exists $\S | \Pi_0$ such that $\tfrac{a}{\Vert \th_\S\Vert^2}=3,4$, respectively.
\item{(2)} When $|\Pi_1|=2$, then both roots in $\Pi_1$ are long; moreover,  for any $\S|\Pi_0$, both roots in $\Pi_1$ are not orthogonal to $\S$. This is {most} easily seen 
by a brief inspection of the untwisted Dynkin diagrams, recalling that, by Section \ref {ta},  $k=1$ and the labels of the roots in $\Pi_1$ in the Dynkin diagram of $\Pia$ are equal to 1.
Anyway, we provide a uniform argument. Let $\Pi_1=\{\a,\be\}$: since  $k=1$ and 
$c_\a(\delta)=1$, $\delta-\a$ is a root and belongs to $\D(\Pia\setminus \{\a\})$. {Since the support of $\delta-\a$ is $\Pia\setminus \{\a\}$, we see that $\Pia\setminus \{\a\}$ is connected.} We claim that $\delta-\a$ is the highest root $\D(\Pia\setminus \{\a\})$. Otherwise, if $\beta>\delta-\a$ and $\beta\in \D(\Pia\setminus \{\a\})$, then  $\beta-\delta$ would be a root with positive coefficients in some simple root in   $\Pia\setminus \{\a\}$ and 
coefficient $-1$ in $\a$. In particular, we obtain that $\delta-\a$ is long with respect to  $\D(\Pia\setminus \{\a\})$ and,  since it has the same length as $\a$, that  
both $\delta-\a$ and $\a$ are long. {For proving the second claim, observe that $\S\cup\{\beta\}\subseteq Supp(\d-\a)=\Pia\setminus\{\a\}$ and the latter is connected.
Hence $\be$  has to be nonorthogonal to $\S$. Switching the role of $\a$ and $\beta$ we get the second claim.}
 \end{rem}
\vskip5pt
Consider the set
$$
D_\s=\bigcup_{w\in W^\s_{ab}}wC_1.
$$
(cf. \eqref{caf}).
 If $\a\in\Da$ then we let $H^+_\a=\{h\in(\h_0)_\R\mid \a(d+h)\ge0\}$. 
The following result refines  \cite[Proposition 4.1]{IMRN}.

\begin{prop}\label{Dsigma}
$$D_\s=\bigcap_{\a\in\Phi_\s}H^+_\a.
$$
\end{prop}

\begin{proof} By  \cite[Propositions 4.1 and 5.8]{IMRN} and by Remark \ref{phisigma} (2), 
we have that  $D_\s=\bigcap\limits_{\a\in\Phi'_\s}H^+_\a
$, where $\Phi'_\s=\Pia_0\cup\{\a+k\d\mid \a\in\Pi_1, \a\text{ long and noncomplex} \}$. {(Actually Propositions 4.1 and 5.8 of \cite{IMRN} cover only the cases when $\g$ is simple, but the argument is easily extended to the adjoint case.)}
 Therefore, we have only to prove that we can restrict from $\Pia_0$ to $\Pia_0^*$, i.e. that if $\S$ is a component of $\Pi_0$ 
such that $a>2\Vert\theta_\S\Vert^2$, then $\theta_\S(x)\leq k$ for all $x\in D_\s$. By Remark \ref{phisigma} (1), $\Pia$ is of type $G_2^{(1)}$ or $A_2^{(2)}$, in particular $\Pi_1$ has a single element: set $\Pi_1=\{\tilde \a\}$. {Note that $\tilde \a$ is long.}
We proceed in steps.
\begin{enumerate}
\item $\tilde\a+3\theta_\S\in\Dap$: this follows from $(\tilde\a,\theta_\S^\vee)<-2$.
\item $2\tilde\a+3\theta_\S\in\Dap_{re}$: indeed  $(\tilde\a,\tilde\a+3\theta_\S)<0$ and $\Vert 2\tilde\a+3\theta_\S\Vert>0$.
\item $k\d-2\tilde\a-3\theta_\S\in \Dp_0$: relation 
$k\d-2\tilde\a-3\theta_\S\in\Da$ follows from (2); it is also clear that it belongs to $\D_0$. So it remains to show that it is positive. Indeed (1) implies $k\d-\tilde\a-3\theta_\S\in\Da$, and this root is positive since $c_{\tilde\a}(k\d-\tilde\a-3\theta_\S)=1$, hence $(k\d-\tilde\a-3\theta_\S)-\tilde\a\in \Dap$.
\end{enumerate}
Now we can conclude, since $(k\d-2\tilde\a-3\theta_\S)(x)\geq 0$ implies $\theta_\S(x)\leq \frac{k}{3}-\frac{2}{3}(\tilde\a,x)\leq k$.
\end{proof}

\begin{rem} In the adjoint case $\g=\k\oplus\k$, $\k$ simple, $D_\s$ is twice the fundamental alcove of the affine Weyl group of $\k$.
\end{rem}

\vskip10pt
We let $\mathcal I_{ab}^\s$ be the set of abelian subalgebras in $\g^{\bar 1}$ that are stable under the action of the Borel subalgebra $\b^{\bar 0}$ of $\g^{\bar 0}$ corresponding to $\Dp_0$. Inclusion turns $\mathcal I_{ab}^\s$ into a poset. 

\begin{prop}\cite[Theorem 3.2]{IMRN}  Let $w\in\Wab$. Suppose $N(w)=\{\beta_1,\ldots,\beta_k\}$. The map $\Wab\to\mathcal I_{ab}^\s$ defined by
$$w\mapsto\bigoplus_{i=1}^k\g^{\bar 1}_{-\ov\beta_i}$$
is a poset isomophism.
\end{prop}

\begin{rem}\label{inL}
The natural isomorphism of $\g^{\bar 0}$-modules $\g^{\bar 1}\cong t^{-1}\otimes \g^{\bar 1}$ { maps the $\b^{\bar 0}$-stable abelian subspaces of $\g^{\bar 1}$ to  $\b^{\bar 0}$-stable abelian subspaces of $\widehat L(\g,\s)$.} Through this isomorphism, the map of the above proposition associates to $w\in \Wab$ the $\b^{\bar 0}$-stable abelian  subalgebra 
$\bigoplus_{i=1}^k\widehat L(\g,\s)_{-\be_i}$.
\end{rem}

\bigskip
Set 
\begin{equation}\label{bw}
\mathcal M_\s=\Phi_\s\backslash (\Pia\cap\Phi_\s).
\end{equation}

\begin{prop}\label{max} If $w\in \Wab$ is maximal, then there is $\a\in \Pia$ and $\mu\in \mathcal M_\s$ such that $w(\a)=\mu$.
\end{prop}

\begin{proof} By Proposition \ref{Dsigma}, we have that, if $\a\in\Pia,\,w(\a)\in\Dap$, then $ws_\a(C_1)\not\subset D_\s$, hence there exists $\mu\in \Phi_\s$ such that $ws_\a(C_1)\not \subset H^+_\mu$. It follows that $\mu\in N(ws_\a)$. Since $N(ws_\a)=N(w)\cup \{w(\a)\}$, we see that $w(\a)=\mu$. We need therefore to prove that there is a simple root $\a$ such that $w(\a)\in \Dap$ and $w(\a)\not \in \Pi_0$. 

Assume on the contrary that, if $\a\in \Pia$ and $w(\a)\in \Dap$, then $w(\a)\in \Pi_0$. 
Then, for all $\alpha\in \widehat \Pi$,
$ht_\sigma(w(\alpha))\leq 0$ and, hence, for all
$\beta\in\widehat \Delta^+$, we have that $ht_\sigma(w(\beta))\leq 0$.  It
follows that, for all $\beta\in\widehat \Delta^+$, if
$w(\beta)$ is positive, then $w(\beta)\in \Delta_0$. Equivalently,
 $w(\widehat\Delta^+)\cap \widehat\Delta^+\subseteq
\Delta_0$.  Hence, in particular, $w(\widehat\Delta^+)\setminus
\widehat\Delta^+$ is infinite, but this is impossible, since
$w(\widehat\Delta^+)\setminus \widehat\Delta^+=-N(w)$.
\end{proof}

\section{The poset $ \mathcal I_{\a,\mu}$ and its minimal elements}
Given $\a\in\Pia,\,\mu\in\mathcal M_\s$, set
$$ \mathcal I_{\a,\mu}=\{w\in \Wab\mid w(\a)=\mu\}.$$
In this Section we find necessary and sufficient conditions for the poset $\mathcal I_{\a,\mu}$ to be nonempty, and in such a case we show that it has minimum. 
\par

We consider first the case $\mu=k\delta-\theta_\S$, with $\S|\Pia_0$. 

\begin{defi}\label{Asigma}
Let $\S|\Pi_0$, and consider the subgraph of $\Pia$ with  $\{\a\in\Pia\mid (\a,\theta_\S)\leq 0\}$ as set of vertices.  We call $A(\S)$ the union 
of the connected components of this subgraph which contain at least one root of $\Pi_1$. Moreover, we set 
$$\G(\S)=A(\S)\cap \S.$$ 
\end{defi}

\begin{rem}\label{ascon}If $|\Pi_1|=1$ then, obviously, $A(\S)$ is connected. If $|\Pi_1|=2$ then a brief inspection shows that there is only one case when $A(\S)$ is disconnected, namely when $\Pia$ is of type $C_n^{(1)}$. Note that in such a case $\Pi_0$ is connected and $\theta_{\Pi_0}$ is a short root. 
\end{rem}

\begin{example}\label{esempio} {\bf (1).} Let $\widehat L(\g,\s)$ be of type $B_n^{(1)}$  $(n\ge 5)$ and $\Pi_1=\{\a_p\},\,4\leq p\leq n-1$. Then $\Pi_0$ has two components, say $\S_1$, of type $D_p$, with simple roots $\{\a_i,\mid 0\leq i\leq p-1\}$, and 
$\S_2$ of type $B_{n-p}$ and  simple roots $\{\a_i,\mid p+1\leq i\leq n\}$. 
We have  $A(\S_1)=\{\a_{p-1},\dots,\a_n\},\ \G(\S_1)=\{\a_{p-1}\}$, and 
$A(\S_2)=\{\a_{0},\dots,$ $\a_{p+1}\}, \ \G(\S_2)=\{\a_{p+1}\}$. We illustrate  this example in the case $n=7,p=4$.

$$
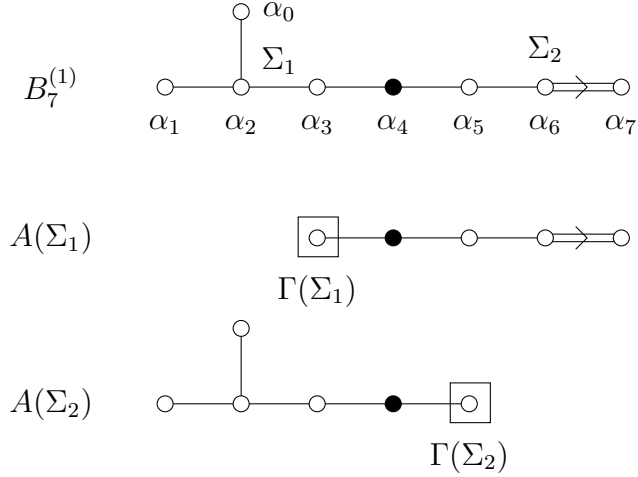
\begin{tikzpicture}[scale=1]
\pgfputat{\pgfxy(-0.5,0)}{\pgfbox[center,center]{$B_7^{(1)}$}}
\pgfputat{\pgfxy(1,-0.5)}{\pgfbox[center,center]{$\a_1$}}
\pgfputat{\pgfxy(2,-0.5)}{\pgfbox[center,center]{$\a_2$}}
\pgfputat{\pgfxy(3,-0.5)}{\pgfbox[center,center]{$\a_3$}}
\pgfputat{\pgfxy(4,-0.5)}{\pgfbox[center,center]{$\a_4$}}
\pgfputat{\pgfxy(5,-0.5)}{\pgfbox[center,center]{$\a_5$}}
\pgfputat{\pgfxy(6,-0.5)}{\pgfbox[center,center]{$\a_6$}}
\pgfputat{\pgfxy(7,-0.5)}{\pgfbox[center,center]{$\a_7$}}
\pgfputat{\pgfxy(2.5,1)}{\pgfbox[center,center]{$\a_0$}}
\pgfputat{\pgfxy(2.5,0.35)}{\pgfbox[center,center]{$\Sigma_1$}}
\pgfputat{\pgfxy(6,0.5)}{\pgfbox[center,center]{$\Sigma_2$}}
\draw{(1,0)--(2,0)--(3,0)--(4,0)--(5,0)--(6,0)};
\draw{(2,0)--(2,1)};
\draw{(6,-0.05)--(7,-0.05)};
\draw{(6,0.05)--(7,0.05)};
\draw{(6.4,.15)--(6.55,0)--(6.4,-.15)};
\draw[fill=black]{(4,0) circle(3pt)};
\draw[fill=white]{(2,1) circle(3pt)};
\foreach \x in {1,2,3,5,6,7}
\draw[fill=white]{(\x,0) circle(3pt)};
\pgfputat{\pgfxy(-0.5,-2)}{\pgfbox[center,center]{$A(\Sigma_1)$}}
\pgfputat{\pgfxy(3,-2.7)}{\pgfbox[center,center]{$\Gamma(\Sigma_1)$}}
\pgfrect[stroke]{\pgfxy(2.75,-2.25)}{\pgfpoint{15pt}{15pt}}
\draw{(3,-2)--(4,-2)--(5,-2)--(6,-2)};
\draw{(6,-2.05)--(7,-2.05)};
\draw{(6,-1.95)--(7,-1.95)};
\draw{(6.4,-1.85)--(6.55,-2)--(6.4,-2.15)};
\draw[fill=black]{(4,-2) circle(3pt)};
\foreach \x in {3,5,6,7}
\draw[fill=white]{(\x,-2) circle(3pt)};
\pgfputat{\pgfxy(-0.5,-4.2)}{\pgfbox[center,center]{$A(\Sigma_2)$}}
\draw{(1,-4.2)--(2,-4.2)--(3,-4.2)--(4,-4.2)--(5,-4.2)};
\draw{(2,-4.2)--(2,-3.2)};
\draw[fill=black]{(4,-4.2) circle(3pt)};
\draw[fill=white]{(2,-3.2) circle(3pt)};
\foreach \x in {1,2,3,5}
\draw[fill=white]{(\x,-4.2) circle(3pt)};
\pgfputat{\pgfxy(5,-4.9)}{\pgfbox[center,center]{$\Gamma(\Sigma_2)$}}
\pgfrect[stroke]{\pgfxy(4.75,-4.45)}{\pgfpoint{15pt}{15pt}}
\end{tikzpicture}
$$
\vskip15pt
\noindent{\bf (2).}
Let $\widehat L(\g,\s)$ be of type $E_6^{(1)}$ and $\Pi_1=\{\a_6\}$. Then $\Pi_0$ has two components: $\S_1$, of type $A_5$, with simple roots $\{\a_1,\ldots,\a_5\}$, and  
$\S_2=\{\a_0\}$, of type $A_1$. We have  $A(\S_1)=\{\a_2,\a_3,\a_4,\a_6,\a_0\},\ \G(\S_1)=\{\a_2,\a_3,\a_4\}$ and $
A(\S_2)=\Pia\setminus\{\a_0\}, \ \G(\S_2)=\emptyset$.
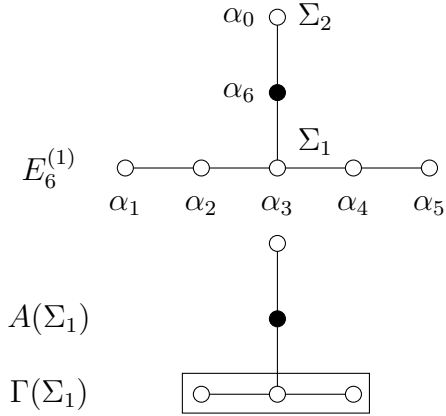
$$
\begin{tikzpicture}[scale=1]
\pgfputat{\pgfxy(-1,0)}{\pgfbox[center,center]{$E_6^{(1)}$}}
\pgfputat{\pgfxy(-1,-3)}{\pgfbox[center,center]{$\Gamma(\Sigma_1)$}}
\pgfputat{\pgfxy(-1,-2)}{\pgfbox[center,center]{$A(\Sigma_1)$}}
\pgfputat{\pgfxy(2.5,0.35)}{\pgfbox[center,center]{$\Sigma_1$}}
\pgfputat{\pgfxy(2.5,2)}{\pgfbox[center,center]{$\Sigma_2$}}
\pgfputat{\pgfxy(1.5,2)}{\pgfbox[center,center]{$\a_0$}}
\pgfputat{\pgfxy(1.5,1)}{\pgfbox[center,center]{$\a_6$}}
\pgfputat{\pgfxy(0,-0.5)}{\pgfbox[center,center]{$\a_1$}}
\pgfputat{\pgfxy(1,-0.5)}{\pgfbox[center,center]{$\a_2$}}
\pgfputat{\pgfxy(2,-0.5)}{\pgfbox[center,center]{$\a_3$}}
\pgfputat{\pgfxy(3,-0.5)}{\pgfbox[center,center]{$\a_4$}}
\pgfputat{\pgfxy(4,-0.5)}{\pgfbox[center,center]{$\a_5$}}
\draw{(2,1)--(2,2)};
\draw[fill=white]{(2,2) circle(3pt)};
{\foreach \x in {0, 1, 2,3}
\draw{(\x,0)--(\x+1,0)};
\draw{(2,0)--(2,1)};
\foreach \x in {0, 1, 2,3,4}
\draw[fill=white]{(\x,0) circle(3pt)};
\draw[fill=black]{(2,1) circle(3pt)};}
\draw{(2,-2)--(2,-1)};
\draw[fill=white]{(2,-1) circle(3pt)};
{\foreach \x in {1, 2}
\draw{(\x,-3)--(\x+1,-3)};
\draw{(2,-3)--(2,-2)};
\foreach \x in {1, 2,3}
\draw[fill=white]{(\x,-3) circle(3pt)};
\draw[fill=black]{(2,-2) circle(3pt)};}
\pgfrect[stroke]{\pgfxy(0.75,-3.25)}{\pgfpoint{70pt}{15pt}}
\end{tikzpicture}
$$
\noindent{\bf (3).}
 Let $\widehat L(\g,\s)$ be of type $A_n^{(1)}$, $(n>2)$, and $\Pi_1=\{\a_0,\a_p\},\,1< p< n$. Then $\Pi_0$ has two components:
$\S_1$, of type $A_{p-1}$,  with simple roots $\{\a_i,\mid 1\leq i\leq p-1\}$, and 
$\S_2$ of type $A_{n-p}$ and  simple roots $\{\a_i,\mid p+1\leq i\leq n\}$. 
We have  and  $A(\S_1)=\S_2\cup \Pi_1$, $A(\S_2)=\S_1\cup\Pi_1$, and 
$\G(\S_i)=\emptyset$ for $i=1,2$. In the following picture we display the case $n=6,p=3$.
$$
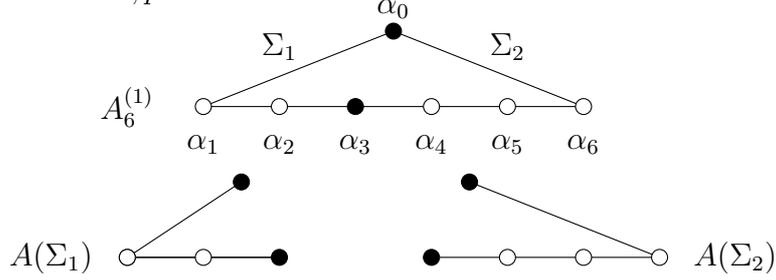
\begin{tikzpicture}[scale=1]
\pgfputat{\pgfxy(-1,0)}{\pgfbox[center,center]{$A_6^{(1)}$}}
\pgfputat{\pgfxy(1,0.8)}{\pgfbox[center,center]{$\Sigma_1$}}
\pgfputat{\pgfxy(4,0.8)}{\pgfbox[center,center]{$\Sigma_2$}}
\pgfputat{\pgfxy(-2,-2)}{\pgfbox[center,center]{$A(\Sigma_1)$}}
\pgfputat{\pgfxy(7,-2)}{\pgfbox[center,center]{$A(\Sigma_2)$}}
\pgfputat{\pgfxy(0,-0.5)}{\pgfbox[center,center]{$\a_1$}}
\pgfputat{\pgfxy(1,-0.5)}{\pgfbox[center,center]{$\a_2$}}
\pgfputat{\pgfxy(2,-0.5)}{\pgfbox[center,center]{$\a_3$}}
\pgfputat{\pgfxy(3,-0.5)}{\pgfbox[center,center]{$\a_4$}}
\pgfputat{\pgfxy(4,-0.5)}{\pgfbox[center,center]{$\a_5$}}
\pgfputat{\pgfxy(5,-0.5)}{\pgfbox[center,center]{$\a_6$}}
\pgfputat{\pgfxy(2.5,1.3)}{\pgfbox[center,center]{$\a_0$}}
{\foreach \x in {0, 1, 2,3,4}
\draw{(\x,0)--(\x+1,0)};
\draw{(0,0)--(2.5,1)};
\draw{(2.5,1)--(5,0)};
\foreach \x in {0, 1,3,4,5}
\draw[fill=white]{(\x,0) circle(3pt)};
\draw[fill=black]{(2.5,1) circle(3pt)};
\draw[fill=black]{(2,0) circle(3pt)};
}
{\foreach \x in {0,1,}
\draw{(\x-1,-2)--(\x,-2)};
\draw{(-1,-2)--(0.5,-1)};
\foreach \x in {0, 1,}
\draw[fill=white]{(\x-1,-2) circle(3pt)};
\draw[fill=black]{(0.5,-1) circle(3pt)};
\draw[fill=black]{(1,-2) circle(3pt)};
}
{\foreach \x in {2,3,4}
\draw{(\x+1,-2)--(\x+2,-2)};
\draw{(3.5,-1)--(6,-2)};
\foreach \x in {3,4,5}
\draw[fill=white]{(\x+1,-2) circle(3pt)};
\draw[fill=black]{(3.5,-1) circle(3pt)};
\draw[fill=black]{(3,-2) circle(3pt)};
}
\end{tikzpicture}
$$

\end{example}

\begin{rem}\label{rsigma}
Assume that $\S|\Pi_0$, $k\d-\theta_\S\in\mathcal M_\s$, $\a\in \Pi_1$, and set 
$$r_\S=-(\a,\theta_\S^\vee).$$ 
By Remark \ref{phisigma} (2), $r_\S$ is independent from the choice of $\a\in \Pi_1$.
Moreover, we see that $r_\S=1$ if and only if $\theta_\Sigma$ is long and non complex while, in the remaining cases, since we are assuming that  $k\d-\theta_\S\in  \Pia^*_0$, we have that $r_\S=2$. If $r_\S=2$, then, for $\a\in\Pi_1$, either $\|\a\|=2\|\theta_\S\|$, or $\ov \a=-\ov \theta_\S$. The latter instance occurs in the adjoint case, so that $k=2$ and  $\theta_\S$ is long and complex. In the first case, 
$\theta_\S$ is a short root,  and $k$ may be 1 or 2. In fact, $k=2$ and $\theta_\S$ is complex, except in the following two cases:  $\g$ is of type $B_n$, $\Pi_1=\{\a_{n-1}\}$ and $\theta_\S=\a_n$ or  $\g$ is of type $C_n$, $\Pi_1=\{\a_0,\a_n\}$, $\S=\{\a_1,\ldots,\a_{n-1}\}$. 
\end{rem}

From now on we will distinguish roots in two types, according to the  following  definition.

\begin{defi}We say that $\a\in\Dap_{re}$ is of type 1 if it  is long and non complex and of type 2 otherwise.
\end{defi}

By the above remark, if $k\d-\theta_\S\in\mathcal M_\s$, its type is $r_\S$.

\begin{lemma}\label{fund}
Assume $\S|\Pi_0$ and $k\d-\theta_\S\in\mathcal M_\s$.
If $\frac{k}{r_\S}\in\ganz$, then $A(\S)$ is connected,  $\frac{k}{r_\S}\d-\theta_\S$ is a root, and\begin{equation}\label{supp}supp\(\frac{k}{r_\S}\d-\theta_\S\)\subseteq A(\S).\end{equation}
\end{lemma}

\begin{proof}
Note that $\frac{k}{r_\S}\in\ganz$ if and only if  $r_\S=1$ or $k=r_\S=2$, in any case $\frac{k}{r_\S}\in\{1,2\}$. If $\frac{k}{r_\S}=2$, then $\frac{k}{r_\S}\d-\theta_\S\in\D$ and, if $\frac{k}{r_\S}=1$ then, either  $k=1$ or $k=2$ and $\theta_\S$ is complex. In both cases, $\frac{k}{r_\S}\d-\theta_\S\in\D$. 
\par
We now prove that $supp(\frac{k}{r_\S}\d-\theta_\S)\subset A(\S)$. Note that $\Pi_1\subset supp(\frac{k}{r_\S}\d-\theta_\S)$, hence we need only to prove that $\a\notin supp(\frac{k}{r_\S}\d-\theta_\S)$ for any $\a\in\S$ such that $(\a,\theta_\S)>0$. We next show that, for such an $\a$, we have $c_\a(\frac{k}{r_\S}\d-\theta_\S)=0$.
We have:
$$
2\frac{r_\S}{k}=-\sum_{\be\in\Pi_1}c_\be(\d)(\be,\theta_\S^\vee)=\sum_{\stackrel{\beta\in\S}{(\beta,\theta_\S)>0}}c_\beta(\d)(\beta,\theta_\S^\vee).
$$
The first equality follows by the definition of $r_\S$, and the second by the relation $(\d,\theta_\S)=0$.
If there is only one root $\a\in\S$ such that $(\a,\theta_\S)>0$, we obtain that 
$$\frac{k}{r_\S}c_\a(\d)(\a,\theta_\S^\vee)=2=c_\a(\theta_\S)(\a,\theta_\S^\vee), 
$$
hence  $c_\a(\frac{k}{r_\S}\d-\theta_\S)=0$.
If there is more than one  root in $\S$ not orthogonal to $\theta_\S$ then $\sum\limits_{\stackrel{\a\in\S}{(\a,\theta_\S)>0}}c_\a(\theta_\S)(\a,\theta_\S^\vee)=2$, hence $(\a,\theta_\S)=c_\a(\theta_\S)=1$ for all $\a\in\S$ not orthogonal to $\theta_\S$. 
\par
Since  $\frac{k}{r_\S}\sum\limits_{\stackrel{\a\in\S}{(\a,\theta_\S)>0}}c_\a(\d)=2$, $\tfrac{k}{r_\S}\in\ganz$, and $c_\a(\d)>0$ for all $\a\in \Pia$, we obtain $\frac{k}{r_\S}c_\a(\d)=1$ and again we have   $c_\a(\frac{k}{r_\S}\d-\theta_\S)=0$, as desired.
\end{proof}
\vskip.5pt
Note that, if $\theta_\S$ is of type 1 or $k=2$, then $\frac{k}{r_\S}\in\ganz$. In particular $A(\S)$ is connected.
\begin{prop}\label{hr}   
Assume $\S|\Pi_0$ and $k\d-\theta_\S\in\mathcal M_\s$.
If  $\theta_\S$ is of type 1, then $k\d-\theta_\S$ is the highest root of $\D(A(\S))$.
If $k=2$ and $\theta_\S$ is of type 2, then $\d-\theta_\S$ is either the highest root of $\D(A(\S))$, or its highest short root.
\end{prop}

\begin{proof} Our assumptions imply in any case that $\frac{k}{r_\S}\in\ganz$. By  \eqref{supp} we have  that  $\frac{k}{r_\S}\d-\theta_\S\in \D(A(\S))$. By the definition of $A(\S)$, $\frac{k}{r_\S}\d-\theta_\S$ is a dominant root in $ \D(A(\S))$, therefore, since $\D(A(\S))$ is a finite root system, we obtain that it is either the highest root of  $ \D(A(\S))$ or its highest short root.  If $\theta_\S$ is of type 1, then it is a long root, so, since $r_\S=1$, $k\d-\theta_\S$ is the highest root of $\D(A(\S))$.
If $\theta_\S$ is of type 2, then $r_\S=2$, hence $\frac{k}{r_\S}=1$. In this case, 
$\theta_\S$ may be short or long, and $\d-\theta_\S$ is the highest short or long root of $\D(A(\S))$, according to its length. 
\end{proof}

\begin{lemma}\label{esse} Assume $\S|\Pi_0$,  $k\delta-\theta_\S\in \mathcal M_\s$ and  $\th_\Sigma$  of type 2. Let $s$ be the element of minimal length in $\Wa$ such that $s(\th_\Sigma)=k\d-\th_\Sigma$. Then $s\in
W(A(\Sigma))$ and is an involution. Moreover,  
$$
N(s)=\{\beta\in \Dap_1\mid (\beta, \theta_\S^\vee)=-2\}, 
$$
in particular, $s\in \Wab$.
\end{lemma}

\begin{proof}
First we assume $k=2$. We claim that in this case $s=s_{\d-\theta_\S}$, which directly implies that it is an involution and, by Proposition \ref{hr}, that it belongs to $W(A(\Sigma))$.	It 
 is immediate that $s_{\d-\theta_\S}(\theta_\S)=2\d-\theta_\S$. Moreover, 
 for each $\a\in\Dap$ which is orthogonal to $\theta_\S$ we have 
$s_{\d-\theta_\S}(\a)=\a\in\Dap$, therefore,  by subsection \eqref{orbit}, $s$ is the unique element of minimal length that maps $\d-\theta_\S$ to $2\d-\theta_\S$. We study $N(s)$.
For each $\beta\in \Dp(A(\S))$,
$$
s(\beta)=\beta+(\beta, \theta_\S^\vee)(\d-\theta_\S)
$$ 
hence $s(\beta)<0$ if and only if $(\beta, \theta_\S^\vee)<0$. Thus if $(\beta, \theta_\S^\vee)=-2$, then $\beta\in N(s)$. It remains to prove the converse. Assume $s(\beta)<0$, hence $(\beta, \theta_\S^\vee)<0$: since  $(\alpha, \theta_\S)\geq 0$ for all
$\alpha\in \Pia\setminus \Pi_1$, this implies that $ht_\s(\beta)\geq 1$. Now we observe that, if $\beta\in N(s)$, then also  $-s(\beta)\in N(s)$, 
therefore $ht_\s(-s(\beta))\geq 1$ as well. Since 
$$ht_\s(s(\beta))=ht_\s(\beta) +(\beta, \theta_\S^\vee)ht_\s(\delta-\theta_\S)=ht_\s(\beta) +(\beta, \theta_\S^\vee),$$ we obtain  that
$-(\beta, \theta_\S^\vee)= ht_\s(\beta)+ht_\s(-s(\beta)) \geq 2$. But  
$k\d-\theta_\S$ belongs to  $\mathcal M_\s\subset \Pi_0^*$, therefore, by \eqref{pi0*}, we have
 $-(\beta, \theta_\S^\vee)\leq \frac{2\Vert \beta\Vert}{\Vert \th_\S\Vert}\leq2\sqrt{2}$, so we can conclude that
$-(\beta, \theta_\S^\vee)=2$ and  $ht_\s(\beta)=ht_\s(-s(\beta))= 1$.
\par
Now we assume $k=1$. By Remark \ref{rsigma}, then either $\g$ is of type $B_n$, $\Pi_1=\{\a_{n-1}\}$ and $\theta_\S=\a_n$, 
or  $\g$ is of type $C_n$, $\Pi_1=\{\a_0,\a_n\}$, $\S=\{\a_1,\ldots,\a_{n-1}\}$.
In the first case a straightforward check shows 
that $s=s_{n-1}\cdots s_2 s_0s_1s_2\cdots s_{n-1}=s_{\a_0+\a_2+\ldots+\a_{n-1}}
s_{\a_1+\a_2+\ldots+\a_{n-1}}$ maps $\a_n$ to $\d-\a_n$, $\a_{n-1}$ to $\a_{n-1}+2\a_n-\d$, fixes $\a_i,\,i=2,\ldots,n-2$ and switches $\a_0$ and $\a_1$. A positive root
$\gamma$ 
is orthogonal to $\a_n$ if and only if $c_{\a_{n-1}}(\gamma)=c_{\a_n}(\gamma)$. Therefore $s$ keeps positive any positive root orthogonal to $\a_n$, as required. 
It is clear that $s$ is an involution, being conjugated to $s_0s_1$. 
A direct computation shows that $N(s)=\{\beta\in\Dp(\Pia\setminus\{\a_n\})\mid c_{\a_{n-1}}(\beta)=1\}=\{\beta\in \Dap_1\mid (\beta, \theta_\S^\vee)=-2\}$. 
\par
For $\g$ of type $C_n$, $s=s_0s_n$ maps $\theta_\S=\a_1+\dots+\a_{n-1}$ to $\d-\theta_\S=\a_0+\dots+\a_{n}$. Moreover, a root in $\Dap$ is orthogonal to $\th_\S$ if and only if 
it is of the form $A\cup(\nat\d\pm A)$ where $A$ is formed by the roots in the subsystem generated by $\a_2,\ldots,\a_{n-2}$ and by the roots
$2\a_i+\ldots+2\a_{n-1}+\a_n,2\leq i\leq n-1$ and $\a_1+\ldots+\a_n$. A direct check shows that these roots are kept positive by $s$, which is therefore minimal.
It is immediate to see that $N(s)=\{\a_0, \a_n\}=\{\beta\in \Dap_1\mid (\beta, \theta_S^\vee)=-2\}$.
\end{proof}
\vskip5pt 

\begin{lemma}\label{technical} Assume $\S|\Pi_0$, $k\d-\theta_\S\in \mathcal M_\s$, $\a\in\Pia$, and  $\Vert\a\Vert=\Vert\theta_\S\Vert$. 
\begin{enumerate}
\item If  $\th_\Sigma$ is of type 1, $\a\in A(\Sigma)$, and  $w_\a$ is the element of minimal length such that $w_\a(\a)= k\d-\th_\Sigma$,  then $w_\a\in\Wab$. 
\item If $\th_\Sigma$ is  of type 2, $\a\in\Sigma$,  $v_\a$ is the element of minimal length in $W(\Sigma)$ such that $v_\a(\a)=\th_\Sigma$,  and $s$ is the element of minimal length in $\Wa$ such that $s(\theta_\S)=k\delta-\theta_\S$, 
then $sv_\a\in\Wab$. Moreover,  $\ell(sv_\a)=\ell(s)+\ell(v_\a)$ and $sv_\a$ is the element of minimal length in $\Wa$ that maps $\a$ to $k\delta-\theta_\S$.
\end{enumerate}
\end{lemma}

\begin{proof} (1). By Proposition \ref{hr} (1) and Proposition \ref{theta} (5), if $\beta\in N(w_\a)$, then there exists $\beta'\in \Dap$ such that $\beta+\beta'= k\d-\th_\Sigma$. By Remark \ref{hts}, each root less than $\mu$ in the usual root order has $\s$-height strictly less than 2, hence $ht_\s(\beta)=ht_\s(\beta')=1$. 
\par  
 (2). Assume first $k=2$, so that $s=s_{\d-\th_\S}$.   
 We first show that $s_{\d-\theta_\S}(\beta)=\beta+\d-\theta_\S$ for each $\beta\in N(v_\a)$. This amounts to prove 
  that $(\theta_\S^\vee,\beta)=1$ for each $\beta \in N(v_{\a})$, which follows again from Proposition \ref{theta}, (2). Thus we obtain that the $\s$-height of the roots in $s_{\d-\theta_\S}(N(v_\a))$ is 1; moreover,  
$$
N(sv_\a)=N(s_{\d-\theta_\S})\cup s_{\d-\theta_\S}(N(v_\a))
$$
and $\ell(sv_\a)=\ell(s)+\ell(v_\a)$.
Since by Lemma \ref{esse},  for each $\beta\in N(s)$,  $ht_\s(\beta)=1$, we conclude that  $sv_\a\in\Wab$.
It remains to prove the assertion about the minimal length. Notice that the above considerations show in particular that, for each $\beta\in N(sv_\a)$, we have that $(\beta, k\d-\theta_\S)\ne 0$. By subsection \ref{orbit}, it follows that
$sv_\a$ is the unique element of minimal length that maps $\a$ to  $k\d-\theta_\S$.
\par
 In the  case of $B_n,$ one has
  $N(sv_\a)=N(s)=\{\beta\in\Dp(\Pia\setminus\{\a_n\})\mid c_{\a_{n-1}}(\beta)=1\}$.
  This follows noting that $L(s)=\{\a_{n-1}\},\ell(s)=2n-2=|\Dp(\Pia\setminus\{\a_{n}\})
  |-|\Dp(\Pia\setminus\{\a_{n-1},\a_{n}\})|$.\par
  In the case $C_n$, we first remark that $sv_{\a_i}=s_0\cdots s_{i-1} s_n\cdots s_{i+1},\,1\leq i\leq n-1$. Thus,
  \begin{align*}N(sv_{\a_i})&=N(s_0\cdots s_{i-1} s_{n}\cdots s_{i+1})\\&=\{\a_0+\ldots+\a_k\mid 0\leq k\leq i-1\}\cup \{\a_h+\ldots+\a_n\mid i+1\leq h\leq n\},\end{align*}
 whose elements have clearly $\s$-height 1. 
The same argument used in case $k=2$ proves that also in this case 
 $sv_\a$ is the unique element of minimal length that maps $\a$ to  $k\d-\theta_\S$.
\end{proof}
\vskip5pt

\begin{lemma}\label{basic}
Assume $\mu\in \mathcal M_\s$, $\a\in \Pia$, and $w\in \mathcal I_{\a,\mu}$. Then
\begin{enumerate}
\item
for each $\beta\in N(w)$, $\mu+\beta\not\in N(w)$;
\item 
for each $\beta, \beta'\in \Dap$ such that $\beta+\beta'=\mu$, exactly one of $\beta$, $\beta'$ belongs to $N(w)$.
\end{enumerate}
\end{lemma}

\begin{proof} (1). We have  
\begin{equation}\label{fff}N(ws_\a)=N(w)\cup\{\mu\}.\end{equation} 
If, for some  $\beta\in N(w)$, $\beta+\mu\in\Dap$, then by the convexity properties, we would obtain $\beta+\mu\in N(w)$: this cannot happen since $ht_\s(\beta+\mu)\geq 3$, while $w$ is $\s$-minuscule. 
\par
(2).
 By the convexity properties, relation \eqref{fff} implies that $N(ws_\a)$ contains at least one summand of each decomposition 
$\mu=\beta+\beta'$, hence $N(w)$ does. Since $\mu\not\in N(w)$, it contains exactly one summand. 
\end{proof}
\vskip5pt

\begin{lemma}\label{minimo}
Assume $\mu=k\d-\theta_\S\in \mathcal M_\s$, $\a\in \Pia$, and $w\in \mathcal I_{\a,\mu}$. Then there exists $u\in \Wa$ such that 
$$\{\beta\in N(w)\mid \mu-\beta\in \Dap\}=N(u).$$
In particular, $u\leq w$. Moreover, $u$ belongs to $\mathcal I_{\a,\mu}$. 
\end{lemma}

\begin{proof}
Set $U=\{\beta\in N(w)\mid \mu-\beta\in \Dap\}$. We first prove the existence of $u$: we have only to check  that $U$ is biconvex. 
We observe that, if $\beta,\beta'\in U$, then $\beta+\beta'$ is not  a root, otherwise it would belong to $N(w)$, which impossible since $ht_\s(\beta+\beta') =2$ and $w$ is $\s$-minuscule. 
Thus we have only to check that, if $\beta\in U$ and $\beta=\gamma+\gamma'$, then at least one of $\gamma$, $\gamma'$ belongs to $U$. Clearly, at least (in fact exactly) 
 one of 
$\gamma,\gamma'$, say $\gamma$, belongs to $N(w)$. We have to prove that $\mu-\gamma$ is a positive root. Set $\beta'=\mu-\beta$: by definition, $\beta'$ is a positive root and it is immediate that  $ht_\s(\beta')=1$. 
Since $\gamma+\gamma'+\beta'=\mu$, at least 
one of $\gamma+\beta'$, $\gamma'+\beta'$, is a root, otherwise, by the Jacobi 
identity,  $\gamma+\gamma'+\beta'$  would not be a root. 
But $\gamma+\beta'$ cannot be a root, otherwise it would have $\s$-height equal to 2, while being less than $\mu$. Therefore $\mu -\gamma=\gamma'+\beta'$ is a root, as required.\par
It remains to prove that $u\in\mathcal I_{\a,\mu}$. It is clear that $u\in\Wab$, 
we have only to check that $u(\a)=\mu$.
By Lemma \ref{basic} (2), $N(w)$ contains exactly one summand of any decomposition of $\mu$ as a sum of two positive roots and, by the definition of $u$, $N(u)$ has the same property. From this fact, we easily deduce that 
$N(u)\cup\{\mu\}$ is biconvex, hence that there exist a simple root
$\beta\in \Pia$ such that $N(us_\beta)=N(u)\cup\{\mu\}$.
But $N(us_\beta)=N(u)\cup\{u(\beta)\}$, hence $u(\beta)=\mu$.
We must prove that $\beta=\a$. 
Since $u\le w$, there exists $z\in \Wa$ such that $w=uz$ and 
$N(w)=N(u)\cup uN(z)$. If $\beta\ne\a$, since $w(\beta)=uz(\beta)\ne\mu$, we obtain that 
$z(\beta)\ne\beta$, hence, by formula \eqref{NW}, that $N(z)$ contains at least one root $\gamma$ such that $\gamma\not \perp \beta$.  Then $u(\gamma)\not \perp\mu$ and $u(\gamma)\in N(w)\setminus N(u)$: we show that this is a contradiction. 
In fact, $u(\gamma)\not \perp\mu$ implies that either $\mu+u(\gamma)$ or $\mu-u(\gamma)$ is a positive root: the first instance is impossible by Lemma \ref {basic} (1); the second one is impossible because it would imply that $u(\gamma)\in N(u)$. 
\end{proof}
\vskip5pt

Assume $\mu=k\delta-\theta_\S$. In Lemma \ref{technical} we  have constructed elements $w_\a$ and $sw_\a$ belonging to $\mathcal I_{\a,\mu}$, under certain restrictions  on $\a$.
In particular, we have proved  that, under such restrictions, $\mathcal I_{\a,\mu}$ is not empty.
 In the next proposition we prove that if $\mathcal I_{\a,\mu}$ is not empty, then $\a$ must satisfy the conditions of Lemma  \ref{technical} (1) (resp. (2)) and the element $u$ built in Lemma \ref{minimo} is actually  
 $w_\a$ (resp. $sv_\a$).  We have therefore determined necessary and sufficient conditions under which   $\mathcal I_{\a,\mu}$ is not empty.
\vskip5pt

\begin{prop}\label{paola} 
 Assume $\S|\Pi_0$, $k\delta-\theta_\S\in \mathcal M_\s$, $\a\in\Pia$, and $w\in \mathcal I_{\a,k\delta-\theta_\S}$. 
\begin{enumerate}
\item
If $\theta_\S$ is of type 1, then $\a\in A(\S)$ and $w\geq w_{\a}$.
\item
If $\theta_\S$ is of type 2, then $\a\in \S$ and $w\geq sv_{\a}$. 
\end{enumerate}
\end{prop}

\begin{proof} 

(1). Set $\mu=k\d-\theta_\S$ and consider the element $u$ built in Lemma \ref{minimo}. By Lemma \ref{basic} (2) and by the definition of $u$, $N(u)$ contains exactly one summand of any decomposition of $\mu$ as a sum of two positive roots, and each element of $N(u)$ is one of the  summands of such a decomposition.  By Proposition \ref{hr} and by Proposition \ref{theta} (6), there exists a simple root $\beta\in A(\S)$ such that 
$u(\beta)=\mu$, and $u$ is the minimal length element with this property. But $u(\a)=\mu$, hence $\a=\beta\in A(\S)$, and $u=w_\a$.  
\par
(2). As above, we set $\mu=k\d-\theta_\S$ and consider the element $u$ built in Lemma \ref{minimo}. We claim that in this case $\a\in \S$ and $u=sv_\a$, which clearly implies the thesis. 
\par
We start proving that $s<u$, which, by Lemma \ref{esse}, consists in proving that all $\beta\in\Dap_1$ such that $(\beta,\mu^\vee)=2$ belong to $N(u)$.
Assume $\beta\in \Dap_1$ and $(\beta,\mu^\vee)=2$: this imply that $\mu-\beta$ and $2\mu-\beta$ are roots, and positive, having positive $\s$-height. By Lemma \ref{basic} (1), 
$\mu-\beta\not\in N(u)$, since $\mu-\beta+\mu$ is a root, hence  $\beta\in N(u)$.
So $s<u$, i.e. there exists $v\in \Wa$ such that 
$u=sv$ and $N(u)=N(s)\cup s N(v)$. It remains to prove that $\a\in \S$ and $v=v_\a$.
\par
First,  we prove that for all $\beta\in N(u)$, we have that $(\beta,\mu^\vee)>0$.  Assume by contradiction that 
$\beta\in N(u)$ and $(\beta,\mu^\vee)=0$, and set $\beta'=\mu-\beta$. 
Then $ht_\s(\beta')=1$ and $(\beta',\mu^\vee)=2$: by the previous part, this implies $\beta'\in N(u)$, which is impossible. Therefore we have 
$(\beta,\mu^\vee)>0$, hence $(\beta,\mu^\vee)\in \{1, 2\}$, since
$\mu\in \mathcal M_\s$. It follows that $s N(v)\subseteq \{\beta\in \Dap_1\mid (\beta,\mu^\vee)=1\}$. 
\par 
Now, we  claim that $N(v)\subseteq \D(\S)$ and that, for each $\beta\in\Dp(\S)$ such that $\theta_\S-\beta$ is positive, exactly one 
among $\beta$ and $\theta_\S-\beta$ belongs to $N(v)$.
Assume $\beta\in N(v)$ and set $\beta'=s(\beta)$. Then $ht_\s(\beta')=1$ and $(\beta',\mu^\vee)=1$, so that $\mu-\beta'$ is a positive root, $(\mu-\beta',\mu^\vee)=1$, and $ht_\s(\mu-\beta')=1$.   By the explicit description of $N(s)$, $\theta_\S-\beta=s(\mu-\beta')$ is positive, hence $\beta\in \D(\S)$. 
Now let $\beta\in\D(\S)^+$ be such that $\theta_\S-\beta\in \D(\S)^+$ and set 
$\beta'=s(\beta)$.  
Then, $\theta_\S$ being long with respect to $\D(\S)$, we obtain that 
$(\beta, \theta^\vee_\S)=(\theta_\S-\beta, \theta^\vee_\S)=1$, hence
$(\beta', \mu^\vee)=(\mu-\beta', \mu^\vee)=1$. Moreover, by the explict description of $N(s)$,  
 both $\beta'$ and $\mu-\beta'$ are positive, therefore both have $\s$-height equal to 1. 
By Lemma \ref{basic} (2), it follows that exactly one among them belongs to $N(u)$, hence to $sN(v)$, therefore, exactly one 
among $\beta$ and $\theta_\S-\beta$ belongs to $N(v)$.
Thus $v$ has the property of Proposition \ref{theta} (6), whence there exists $\beta\in \S$ such that $v=v_\beta$. But $s v_\beta(\beta)=\mu$, hence $\beta=\a$. 
\end{proof}
\vskip5pt

We have finally to deal with the posets $\mathcal I_{\a,\mu}$ with $\mu=\be+k\delta$,  $\be\in\Pi_1$. According to our definitions, 
\eqref{phis} and \eqref{bw}, the assumption that $\be+k\delta\in \mathcal M_\s$
implies that $\be$ is long.
\par

We start  refining the analysis done in  \cite[Lemma 5.10]{IMRN}.

\begin{prop}\label{om}
If $\g^{\bar 0}$ is semisimple, then  $\g^{\bar 1}$ is irreducible as a $\g^{\bar 0}$-module. If $\g^{\bar 0}$ is not semisimple, then  $\g^{\bar 1}$ has two irreducible components as a $\g^{\bar 0}$-module. As a consequence, the following 
holds. Denote by $w_0$ the longest element of  $W_0$. Then
\begin{enumerate} 
\item if $\Pi_1=\{\a\}$   then $w_0(\a)=k\d-\a$;
\item 
if $\Pi_1=\{\a,\be\}$, then 
$w_0(\a)=\d-\be$ and $w_0(\be)=\d-\a$.
\end{enumerate}
\end{prop}

\begin{proof} It is well-known that  $t^{-1}\otimes \g^{\bar 1}$ occurs as a submodule of the homology $H_1(\mathfrak u^-)$ where $\mathfrak u^-=\sum_{\a\in(-\Dap)\backslash\D_0}\widehat L(\g,\s)_\a$. By Garland-Lepowsky theorem, this homology  decomposes as  $\oplus_{\a\in\Pi_1}V(-\a)$,
as a sum of irreducible ${(\g^{\bar0}+\C K+\C d)}$-modules, which stay irreducible as $\g^{\bar0}$-modules. It follows that
\begin{equation}\label{dec1}
t^{-1}\otimes \g^{\bar1}=V(-\a), \text{ if } \Pi_1=\{\a\}.
\end{equation}
Moreover, it is clear that $-\a$ occurs as a highest weights of $t^{-1}\otimes \g^{\bar 1}$, for any  $\a\in\Pi_1$, hence,  
\begin{equation}\label{dec2}
t^{-1}\otimes \g^{\bar 1}=V(-\a)\oplus V(-\be), \text{ if } \Pi_1=\{\a,\beta\} \text{ with $\a\ne\be$.}
\end{equation}

Since  $\g^{\bar 1}$ is self-dual as a $\g^{\bar 0}$-module, 
 if  $\Pi_1=\{\a\}$ we obtain that $w_0(\bar \a)=-\bar\a$, hence
$w_0(\a)=w_0(\delta'+\bar \a)=\delta'-\bar\a=2\delta'-\a=k\delta-\a$
(cf. Section \ref{ta}), as claimed.
\par 

Assume $\Pi_1=\{\a, \be\}$. Notice that in this case $k=1$, so that $\delta=2\delta'$ and that $c_\a(\delta)=
c_\be(\delta)=1$ (see Section \ref{ta}).
  We have two cases:
\begin{enumerate}
\item $V(-\a)^*=V(-\a)$ and  $V(-\be)^* = V(-\be),$
\item $V(-\a)^*=V(-\be)$ and  $V(-\be)^* = V(-\a).$
\end{enumerate}

In the first case we have  $w_0(\bar \a)=-\bar\a$, which forces
$w_0(\a)=w_0(\d'+\bar\a)=\d'-\bar\a=\d-\a$ and this is not possible since
 $c_\a(w_0(\a))=c_\a(\a)=1$, while $c_\a(\d-\a)=0$. Hence (2) holds. It follows that 
$w_0(\bar\a)=-\bar\be$ and  $w_0(\bar\be)=-\bar\a$. Therefore, 
$w_0(\a)=\d'-\bar\be=2\d'-\be=\d-\be$ and
$w_0(\be)=\d'-\bar\a=2\d'-\a=\d-\a$.
\end{proof}
\vskip5pt

\begin{prop}\label{piofap} Assume $\mu=\a+k\delta\in \mathcal M_\s$, with $\a\in \Pi_1$. 
Set $\Pi_{0,\a}=\Pi_0\cap\a^\perp$, $W_{0, \a}=W(\Pi_{0,\a}),$
and denote by $w_{0, \a}$ the longest element of $W_{0, \a}$. 
\begin{enumerate} 
\item If $\Pi_1=\{\a\}$, then $\mathcal I_{\gamma,\mu}\ne\emptyset$ if and only if  $\gamma=\a$. Moreover, 
$$\mathcal I_{\a,\mu}=\{s_\a w_{0, \a}w_0\}.$$
\item If $\Pi_1=\{\a,\be\}$, then $\mathcal I_{\gamma,\a+k\d}\ne\emptyset$ if and only if  $\gamma=\be$.
Moreover, 
$$\min \mathcal I_{\be,\a+k\d}=s_\a w_{0, \a}w_0.$$
\end{enumerate}
\end{prop}

\begin{proof} 

Set $x=s_\a w_{0,\a}w_0$. By Proposition \ref{om}, we have that:
\item{(1)} if $\Pi_1=\{\a\}$, then 
$x(\a)=s_\a w_{0,\a}w_0(\a)=s_\a w_{0,\a}(k\delta-\a)=s_\a (k\delta-\a)= k\delta+\a$;
\item {(2)} if $\Pi_1=\{\a, \be\}$, then 
$x(\be)=s_\a w_{0,\a}w_0(\be)=s_\a w_{0,\a}(k\delta-\a)=s_\a (k\delta-\a)= k\delta+\a$.
\par

We prove that $x$ is $\s$-minuscule.
Since  $w_{0,\a}w_0\in W_0$, it is clear that $N(w_{0,\a}w_0)\subseteq \Dp_0$. In fact, we have $N(w_{0,\a}w_0)=\D_0^+\setminus \D(\Pi_{0, \a})$. Since $\a$ is long, for each $\gamma\in \D_0^+\setminus \D(\Pi_{0, \a})$, we have 
$s_\a(\gamma)=\gamma+\a$, hence  $N(x)=\{\a\}\cup s_\a N(w_{0,\a}w_0)\subseteq \Da_1$, as claimed.
\par
So we have proved that 
$x\in \mathcal I_{\a,\a+k\delta}$, if $\Pi_1=\{\a\}$, and $x\in \mathcal I_{\be,\a+k\delta}$, if $\Pi_1=\{\a, \be\}$.
\par
Now we treat separately the two cases.
First, let $\Pi_1=\{\a\}$ and assume that $w\in \mathcal I_{\gamma,\a+k\delta}$, with $\gamma\in \Pia$. Then $N(ws_\gamma)=N(w)\cup\{\a+k\delta\}$, hence, since $w$ is $\sigma$-minuscule, 
$$w(C_1)\subseteq \bigcap_{\eta\in\Pia_0}H^+_\eta\, \setminus\,  
\bigcap_{\eta\in \Phi_\s}H^+_\eta= P_\s\setminus D_\s,$$ 
where we denote by $P_\s$ the polytope $\bigcap_{\eta\in\Pia_0}H^+_\eta$.
But by \cite[Lemma 5.11]{IMRN}, there is exactly one $w\in \Wa$ such that 
$w(C_1)\subseteq P_\s\setminus D_\s$, hence $w=x$, $\gamma=\a$, and 
$\mathcal I_{\a,\a+k\delta}=\{x\}$.
\par
Now we assume $\Pi_1=\{\a, \be\}$, $\gamma\in\Pia$ and  $w\in \mathcal I_{\gamma,\d+\a}$. We will show that $\gamma=\be$ and $x\le w$. 
By Remark \ref{phisigma} (2) both roots in $\Pi_1$ are long; moreover,  
$\delta-\a$ is the highest root of $\D(\Pia\setminus\{\a\})$. For any $\gamma\in 
\Pia\setminus\{\a\}$, let $v_\gamma$ be the element of minimal length that maps 
$\gamma$ to $\delta-\a$. We start proving that $w_{0,\a}w_0=v_\be$. In fact, it is clear that $w_{0,\a}w_0(\be)=\d-\a$, so it suffices then to check that $(w_{0,\a}w_0)^{-1}(\gamma)>0$ for all $\gamma\in \Pia$ such that  $(\a,\gamma)=0$. If $\gamma\in \Pi_{0,\a}$ then $(w_{0,\a}w_0)^{-1}(\gamma)=w_0w_{0,\a}(\gamma)>0$. Moreover, in any case $(w_{0,\a}w_0)^{-1}(\be)>0$,  since $N(w_{0,\a}w_0)\subset\D_0$. Thus we obtain $w_{0,\a}w_0=v_\be$, 
 $x=s_\a v_{\be}$, and $N(x)=\{\a\}\cup s_\a(N(v_{\be}))$.
\par 
Now we consider $w$. Since $w(\gamma)=\d+\a$, we have $w^{-1}(\a)=-\d+\gamma$ hence $\a\in N(w)$. It follows that $w=s_\a z$ with $\ell(w)=1+\ell(z)$. In particular, $N(w)=\{\a\}\cup s_\a(N(z))$. Since $z(\gamma)=\delta-\a$, we have that $N(zs_\gamma)=N(z)\cup\{\d-\a\}$, so the biconvexity of   $N(zs_\gamma)$ implies that 
for any pair $\eta_1,\eta_2\in \Dap$
such that $\eta_1+\eta_2=\d-\a$ exactly one of $\eta_1,\eta_2$ belong to $N(z)$.
Moreover, $\d-\a$ being a long root, for any such pair of roots we have $(\eta_i, (\d-\a)^\vee)=1$, for $i=1, 2$, since  $(\eta_1+\eta_2, (\d-\a)^\vee)=2$
and  $(\eta_i, (\d-\a)^\vee)\leq 1$, for $i=1, 2$. It follows that $s_\a(\eta_i)=\eta_i+\a$ and therefore, that $ht_\s(s_\a(\eta_i))=ht_\s(\eta_i)+1$, 
for $i=1,2$. Now, if $\eta_i\in N(z)$, $s_\a(\eta_i)\in N(w)$, and we obtain that 
$ht_\s(\eta_i)=0$. But $ht_\s(\d-\a)=1$, so that one of the $\eta_i$ has $\s$-height equal to 1 and the other has $\s$-height equal to 0. This implies that 
for any pair $\eta_1,\eta_2\in \Dap$
such that $\eta_1+\eta_2=\d-\a$, $N(z)$ contains exactly the summand  $\eta_i$
having  $\s$-height equal to 0. This must hold in particular when we take $w=x$ and so $z=v_\be$. In this case we clearly obtain that $N(v_\be)$ is exactly the set of the summands of $\s$-height equal to 0 of all the decomposition of $\d-\a$ as a sum of two positive roots. So, for a general $w$, we obtain that
$N(v_\be)\subseteq N(z)$, whence $w_{\be,\a+\d}=s_\a v_{\be}\le s_\a z= w$ as desired.\par 
It remains to prove that $\gamma=\beta$. We have $z=v_\be y$ with $N(z)=N(v_\be)\cup v_\be N(y)$, and $y(\gamma)=\be$. If $\gamma\ne\be$, then $N(y)$ would contain some roots not orthogonal to $\be$, whence
$v_\be N(y)$ contains some root $\eta$ not orthogonal to $\d-\a$, hence to $\a$. 
It follows that $s_\a(\eta)=\eta\pm\a\in N(w)$. But $\eta-\a\not\in N(w)$, being summable to $\a$ that belongs to $N(w)$, hence $s_\a(\eta)=\eta+\a\in N(w)$.
In particular, $ht_\s(\eta)=0$, and $\d-\a-\eta\in \Dap$: this implies that 
$\eta\in N(v_\be)$, a contradiction.
\end{proof}

We sum up the results we have obtained in the following theorem. 
\vskip5pt

If $S$ is a connected subset of the set of simple roots, we denote by $S_\lu$ the set of elements of $S$ of the same length of $\theta_{S}$. It is clear that, with respect to $\D(S)$, $\theta_{S}$ is a long root, therefore $S_\lu$,  is the set of the long roots of $S$, with respect to the subsystem $\D(S)$.
With notation as in Lemma \ref{technical} and Proposition \ref{piofap}, we set 
\begin{equation}\label{wamu}
w_{\a,\mu}=\begin{cases}
w_\a&\text{if $\mu=k\d-\theta_\S$, $\theta_\S$ is of type 1, and $\a\in {A(\S)}_\lu$}\\
sv_\a&\text{if $\mu=k\d-\theta_\S$,  $\theta_\S$ is of type 2, and $\a\in\S_\lu$}\\
s_\be w_{0,\be}w_0 &\text{if $\mu=\be+k\delta$, $\be\in \Pi_1$}
\end{cases}
\end{equation}
and
\begin{equation}
\Pia_{\mu}=\begin{cases}
A(\S)_\lu & \text{if $\mu=k\d-\theta_\S$ and $\theta_\S$ is of type 1}\\
\S_\lu & \text{if $\mu=k\d-\theta_\S$ and  $\theta_\S$ is of type 2}\\
\Pi_1 & \text{if $\mu=\be+k\delta$ and $\{\be\}= \Pi_1$}\\
\Pi_1\setminus \{\be\} & \text{if $\mu=\be+k\delta$, $\be\in \Pi_1$, and $|\Pi_1|=2$}
\end{cases}
\end{equation}
\vskip5pt

\begin{theorem}\label{min} Assume $\mu\in \mathcal M_\s$ and $\a\in \Pia$.
Then $ \mathcal I_{\a,\mu}\ne \emptyset$ if and only if $\a\in \Pia_\mu$.
Moreover,
$$w_{\a,\mu}=\min\, \mathcal I_{\a,\mu}.$$
\end{theorem}
\begin{proof} 
The claim follows directly from Lemma \ref{technical}, Proposition \ref{paola}, and Proposition \ref{piofap}.
\end{proof}

\section{The  {poset} structure of $\mathcal I_{\a,\mu}$}
{We now study the poset structure of the sets $\mathcal I_{\a,\mu}$. This study is motivated by  the following result, that shows that the maximal elements of the sets $\mathcal I_{\a,\mu}$ are maximal in the whole poset $\Wab$
except when $\a\in\Pi_1$ and $\mu=k\d-\theta_\S,\,\S|\Pi_0$.} 
\begin{prop}\label{fuori} Suppose $w\in\mathcal I_{\a,\mu}$ and  $v\ge w$ with $v\in\Wab$.   If  $v\notin\mathcal I_{\a,\mu}$, then $\a\in\Pi_1$. In that case, write explicitly 
$\Pi_1=\{\a,\be\}$ (with $\be=\a$ if $|\Pi_1|=1$).
Then  $v\in\mathcal I_{\a,k\d+\be}$.
\end{prop}
\begin{proof}
If $v\notin \mathcal I_{\a,\mu}$, write $v=wxs_\gamma y$ with $wx\in\mathcal I_{\a,\mu}$, $wxs_\gamma\notin \mathcal I_{\a,\mu}$ and $\ell(v)=\ell(w)+\ell(x)+\ell(y)+1$. Then $(\gamma,\a)<0$.
Set $(\a,\gamma^\vee)=-r$ and consider $wxs_\gamma s_\a$. We have 
$$N(wxs_\gamma s_\a)=N(wxs_\gamma)\cup\{wx(\a+r\gamma)\}=N(wxs_\gamma)\cup\{\mu+ rwx(\gamma)\}.$$
Note that $ht_\s(\mu+rwx(\gamma))=ht_\s(\mu)+r$. Since the latter root is not simple, there exists $\eta\in\Pia$ such that $\mu+wx(\gamma)-\eta\in\Dap$. 
Since  $N(wxs_\gamma)\subset\Da_1$ and $N(wxs_\gamma s_\a)$ is convex, we have that $\eta\notin\Pi_0$.   
Hence $\mu+rwx(\gamma)$ is minimal in $\Da_{ht_\s(\mu)+r}$. Now we  use  Remark \ref{hts} about minimal roots. If $ht_\s(\mu)+r=2s$ with $s>1$ then $\mu+rwx(\gamma)=ks\d-\theta_\S$ for some $\S|\Pi_0$. But then, by convexity, $k\d-\theta_\S\in  N(wxs_\gamma)$ which is absurd. If $ht_\s(\mu)+r=2s+1$ with $s>1$ then $\mu+rwx(\gamma)=ks\d+\be$ for some $\be\in\Pi_1$. But then, by convexity, $k\d+\be\in  N(wxs_\gamma)$ which is absurd. Therefore $ht_\s(\mu)=2$ and $r=1$. It follows that there exists $\beta\in\Pi_1$
 such that $\mu+wx(\gamma)=\beta+k\d$.
  In turn, we deduce that  $wxs_\gamma\in \mathcal I_{\a,k\d+\be}$. By Proposition \ref{piofap}, (1), we have $\a\in\Pi_1$ as claimed, and $wxs_\gamma\in \mathcal I_{\a,k\d+\be}$ with $\be=\a$ if $|\Pi_1|=1$ and $\Pi_1=\{\a,\be\}$ otherwise. Since  $v\ge wxs_\gamma\in\mathcal I_{\a,k\d+\be}$ and $ht_\s(k\d+\be)=3$, by the first part of the proof, we have that $v\in \mathcal I_{\a,k\d+\be}$, as claimed.   
\end{proof}

{
We now turn to the description of the poset structure of $\mathcal I_{\a,\mu}$: we will show that it is isomorphic to a poset $G'\backslash G$ for suitable reflection subgroups 
$G,G'$ of $\Wa$.}

 \begin{defi}\label{Pialfa} For $\a\in\Pia$, and $\S|\Pi_0$, we  set 
$$\Pia_\a=\Pia\cap\a^\perp, \quad
\Wa_\a=W(\Pia_\a).$$
\end{defi}

\begin{lemma}\label{orto}
Let  $\mu\in\mathcal M_\s$, $u, v\in \mathcal I_{\a,\mu}$, and $u<v$. Then $v=ux$ with $x\in \Wa_\a$. In particular,
$$ \mathcal I_{\a,\mu}\subseteq w_{\a, \mu}\Wa_\a.$$
\end{lemma}

\begin{proof} By assumption, there exists $x\in \Wa$ such that 
$N(v)=N(u)\cup uN(x)$: suppose by contradiction that $x\notin\Wa_\a$. 
Then we may assume $x=x_1s_{\beta}x_2$ with $\ell(x)=\ell(x_1)+\ell(x_2)+1$, $x_1\in\Wa_\a$,  and $\beta\in\Pia$, $\beta\not\perp\a$. Then $N(ux_1)\cup ux_1(\beta)\subseteq N(v)$. But $(\beta,\a)<0$, hence  $(u x_1(\beta),u x_1(\a))=(u x_1(\beta),\mu)<0$, therefore $ux_1(\beta)+\mu$ is a root: this cannot happen
by Lemma \ref{basic} (1).
\end{proof}

By Lemma \ref{orto}, $\mathcal I_{\a,\mu}$ is in bijection,  in a natural way, with a subset of $\Wa_\a$, namely, the subset of all $u\in \Wa_\a$ such that $w_{\a, \mu}u\in 
\mathcal I_{\a,\mu}$. We will show that this subset is a system of minimal coset representatives of $\Wa_\a$ modulo a certain subgroup $\Wa_{\a,\mu}$. This will take the rest of the Section.
\par
We start with giving a combinatorial characterization of the elements $u$ such that 
 $w_{\a, \mu}u\in \mathcal I_{\a,\mu}$.

\begin{defi}\label{Valfa-Bmu}
We set 
\begin{align*}\\
B_\mu&=
\begin{cases}
\{\gamma\in\Pia\mid (\gamma,\theta_\S^\vee)=1\} & \text {if  }\mu=k\d-\theta_\S  \text { and }\theta_\S \text{ is of type 1,}\\
\Pi_1 &  \text{if }\mu=k\d-\theta_\S  \text{ and }\theta_\S  \text{ is of type 2},\\
\{\beta\} &  \text{if }\mu=k\d+\be,\ \be\in \Pi_1,\\
\end{cases} \\\\
V_{\a,\mu}&=
\{w\in\Wa_\a\mid ht_{B_\mu}(\gamma)=1\ \forall\,\gamma\in N(w)\},
\end{align*}
if $B_\mu\ne\emptyset$.  If $B_\mu=\emptyset$, we set $V_{\a,\mu}=\{{1}\}$.
\end{defi}
\vskip5pt

\begin{lemma}\label{Bsigma}
Assume $\S|\Pi_0$, $\mu=k\d-\theta_\S\in \mathcal M_\s$, and set  
$$
B_\S= \{\gamma\in\Pia\mid (\gamma,\theta_\S)>0\}.
$$
Then
\begin{enumerate}
\item 
For  all $\eta\in \Da$, 
\begin{equation}\label{exP}(\eta,\mu^\vee)=ht_\s(\eta)r_\S -ht_{B_\S}(\eta)\varepsilon_\S,\end{equation}
 where  $\varepsilon_\S=2$, if   $|\S|=1$, $\varepsilon_\S=1$, otherwise.
\item
$B_\S=\{\gamma\in\Pia\mid (\gamma,\theta_\S^\vee)=1\}$, unless $|\S|=1$.
\end{enumerate}
\end{lemma}

\begin{proof}
It is clear that for $\gamma\in \Pia$, we have $(\gamma, \theta_\S)<0$ if and  only if $\gamma\in \Pi_1$; moreover, recall that $r_\S=-(\gamma, \theta_\S^\vee)=(\gamma, \mu^\vee)$ for all $\gamma\in \Pi_1$.  \par
On the other hand, by definition, we have $(\gamma, \theta_\S)>0$ if and only if $\gamma\in B_\S$. Clearly, $B_\S\subseteq \S$, and since $\theta_\S$ is long with respect to $\D(\S)$, it follows that, if 
$\gamma\in B_\S$, then  $(\gamma,  \theta_\S^\vee)=1$ unless $\S=B_\S=\{\theta_\S\}$, 
in which case  $(\gamma,  \theta_\S^\vee)=2$.
Therefore, for any $\eta\in\Da$, 
$$(\eta,\mu^\vee)=\sum\limits_{\gamma\in \Pi_1}c_\gamma(\eta)(\gamma,\mu^\vee)+\sum\limits_{\gamma\in B_\S}c_\gamma(\eta)(\gamma,\mu^\vee)=
ht_\s(\eta)r_\S-ht_{B_\S}(\eta)\varepsilon_\S$$
as wished.\end{proof}

\begin{lemma}\label{Valfa} Assume $\mathcal I_{\a,\mu}\ne\emptyset$.
For any $u\in\Wa$,  $w_{\a,\mu}u\in\mathcal I_{\a,\mu}$ if and only if $u\in\ V_{\a,\mu}$.
\end{lemma}
\begin{proof}
We deal with the three cases that occur in the definition of $B_\mu$ one by one.  We shall use several times relation \eqref{exP} from Lemma
\ref{Bsigma}.
\medskip

\item{\it 1.} $\mu=k\d-\theta_\S$, $\theta_\S$ of type 1.  Then $\a\in A(\Sigma)$,  $\mu$ is the highest root of $A(\Sigma)$, and $w_{\a,\mu}\in W(A(\S))$.  
It is clear that $B_\S\cap A(\Sigma)=\emptyset$, in fact, by Definition \ref{Asigma}, $A(\Sigma)$ is a connected component of $\Pia\setminus B_\S$.
In particular, for all $\eta\in\Da$, $ht_{B_\S}(w_{\a,\mu}(\eta))=
ht_{B_\S}(\eta)$.  Recall that $r_\S$ is the type of $\theta_\S$.
By  \eqref{exP}, 
\begin{align*} 
(w_{\a,\mu}(\eta),\mu^\vee)&=ht_\s(w_{\a,\mu}(\beta))-ht_{B_\S}(w_{\a,\mu}(\eta))\varepsilon_\S=\\
&=ht_\s(w_{\a,\mu}(\beta))-ht_{B_\S}(\eta)\varepsilon_\S.
\end{align*}
Now, assume  $u\in V_{\a,\mu}$. If $u={1}$, obviously 
$w_{\a, \mu}u\in  \mathcal I_{\a, \mu}$. So we may assume $u\ne {\bf 1}$ and $|\S|>1$. If $\eta\in N(u)$, then  $(\eta,\alpha)=0$, so that  $(w_{\a,\mu}(\eta),\mu)=0$; moreover, $ht_{B_\S}(\eta)=\varepsilon_\S=1$. Therefore, by the above identities we obtain that $ht_\s(w_{\a,\mu}(\eta))=1$. 
Thus  $N(w_{\a,\mu}u)=N(w_{\a,\mu})\cup w_{\a,\mu}N(u)\subseteq \Da_1$,  hence $w_{\a,\mu}u\in \Wab$. Since $u(\a)=\a$, we conclude that  $w_{\a,\mu}u\in \mathcal I_{\a, \mu}$. 
\par
Conversely, if $w_{\a,\mu}u\in \mathcal I_{\a, \mu}$ with $u\ne 1$, then, by Lemma \ref{orto}, $u\in \Wa_\a$, so that, if $\eta\in N(u)$, then $(\eta, \a)=0$, hence $(w_{\a,\mu}(\eta),\mu)=0$. Moreover, $ht_\s(w_{\a,\mu}(\eta))=1$. It follows that $\varepsilon_\S=1$ and $ht_{B_\S}(\eta)=1$, so  $ht_{B_\S}(\eta)=ht_{B_\mu}(\eta)=1$, hence  $u\in V_{\a, \mu}$.
\medskip

\item{\it 2.} $\mu=k\d-\theta_\S$, $\theta_\S$ of type 2.  Then $\a\in\Sigma$,  and $w_{\a,\mu}=s v_\a$, where $v_\a$ is the minimal element that maps $\a$ to $\theta_\S$ and $s$ is the minimal element that maps $\theta_\S$ to $\mu$.   We also know that $s$ is an involution.
In this case, $B_\mu=\Pi_1$, 
hence $B_\mu\cap \Sigma=\emptyset$. Thus the $B_\mu$-height 
is the $\s$-height and, since $v_\a\in W(\S)$,   we have that  $v_\a$ preserves the $\s$-height. 
Similarly, since $s\in W(A(\S))$, $s$ preserves the $B_\S$-height.
Therefore, 
for all $\eta\in\Da$, we obtain that
\begin{align*} 
(w_{\a,\mu}(\eta),\mu^\vee)&=( v_\a(\eta),s\mu^\vee)
=( v_\a(\eta),\theta_\S^\vee)=-( v_\a(\eta),\mu^\vee)\\
&=-2 ht_\s(v_{\a}(\eta))+ht_{B_\S}(v_{\a}(\eta))\varepsilon_\S\\
&=-2 ht_\s(\eta)+ht_{B_\S}(v_{\a}(\eta))\varepsilon_\S,
\end{align*}
and also that
\begin{align*}
(w_{\a,\mu}(\eta),\mu^\vee)
&=2 ht_\s(w_{\a,\mu}(\eta))-ht_{B_\S}(w_{\a,\mu}(\eta))\varepsilon_\S\\
&=2 ht_\s(w_{\a,\mu}(\eta))-ht_{B_\S}(v_\a(\eta))\varepsilon_\S.
\end{align*}
In particular, if $(\mu^\vee,w_{\a,\mu}(\eta))=0$, then $ ht_\s(w_{\a,\mu}(\eta))=ht_\s(\eta)=ht_{B_\mu}(\eta).$ By Lemma \ref{orto}, this directly implies that $w_{\a,\mu}u\in \mathcal I_{\a, \mu}$ if and only if $u\in V_{\a,\mu}$. 
\medskip

\item{\it 3.} $\mu=k\d+\be,\ \be\in \Pi_1$.
If $|\Pi_1|=1$, then $V_{\a,\mu}=\{{1}\}$ and, by
Proposition \ref{piofap}, $\mathcal I_{\a,\mu}=\{w_{\a,\mu}\}$. So we may assume $|\Pi_1|=2$, $\Pi_1=\{\a,\be\}$. Then, with notation as in Proposition \ref{piofap},  we have that  $w_{\a,\mu}=s_\be v_{\be}$. Since $v_{\be}(\a)=\d-\be$, we deduce that $v_{\be}^{-1}(\be)=\d-\a$, hence, if $(\gamma,\a)=0$, then
$$
s_\be v_{\be}(\gamma)=v_{\be}(\gamma)-(v_{\be}(\gamma),\be^\vee)\be=v_{\be}(\gamma)-(\gamma,\d-\a^\vee)\be=v_{\be}(\gamma).
$$
It follows that, if $\gamma\in\Dap_\a$, then  $ht_{\s}(w_{\a,\mu}(\gamma))=ht_\s(v_{\be}(\gamma))=c_{\be}(\gamma)=ht_{B_\mu}(\gamma)$ and we can argue as in case 2.
\end{proof}

\begin{lemma}\label{M2} 
Assume $\a\in \Pia$, $\mu\in \mathcal M_\s$, 
$\mathcal I_{\a,\mu}\ne\emptyset$, $B_\mu\ne\emptyset$,
and set 
$$\D^2_{\a,\mu}=\{\gamma\in\D(\Pia_\a)\mid  ht_{B_\mu}(\gamma)\geq 2 \}.$$
Then 
$\D^2_{\a,\mu}\ne\emptyset$ if and only if 
$\mu=k\d-\theta_\S$ with $\theta_\S$ of type $1$ and  $|\S|>1$, and
$\a\in A(\S)\setminus (\S\cup\Pi_1)$.
In this case, $\theta_\S$ is the minimal element in  $\D^2_{\a,\mu}$, with respect to the usual root order. Moreover,  $ht_{B_\mu}(\th_\S)=ht_{B_\S}(\th_\S)=2$.
\end{lemma}

\begin{proof}
We deal with the three cases that occur in the definition of $B_\mu$ one by one.
\medskip

\item{\it 1.} $\mu=k\d-\theta_\S$, $\theta_\S$ of type 1.  Then $\a\in A(\Sigma)$ and $|\S|>1$, since we are assuming $B_\mu\ne\emptyset$. In particular $B_\S=B_\mu=\{\gamma\in\Pia\mid(\gamma,\theta_\S^\vee)=1\}$.

We first prove that if  $\gamma\in\D(\Pia_\a)$ and $ht_{B_\mu}(\gamma)\geq 2$ then $\gamma\ge\theta_\S$. We notice  $(\be,\theta_{\S}^\vee)\in\{0,1\}$ for any $\be\in\Dp(\S)\backslash\{\theta_{\S}\}$, therefore, since $(\theta_{\S},\theta_{\S}^\vee)=2$, $ht_{B_{\mu}}(\th_\S)=2$ and $\th_\S$ is the unique root in $\D(\S)$ with this property. It follows that we can assume
$\gamma\notin\D(\S)$, so that $ht_\s(\gamma)>0$. 
 Since
 $c_\a(k\d-\gamma)>0$, we have that $k\d-\gamma$ is a positive root,
hence
 $ht_\s(\gamma)\le 2$. If $ht_\s(\gamma)=1$, then $(\gamma,\th_{\S}^\vee)=1$, hence 
 $\gamma-\th_{\S}$ is a root, which can't be negative, since $\gamma$ is supported also outside $\S$. So it is positive, hence $\gamma\geq\th_{\S}$. Suppose now $ht_\s(\gamma)=2$.  Then $k\d-\gamma\in\D_0$, hence it should belong to  the component $\S'$ of $\Pi_0$ to which $\a$ belongs, since $c_\a(k\d-\gamma)>0$. Thus $\gamma=k\d-\be$ with $\be\in \S'$. If $\S=\S'$, then $\a\in\Gamma(\S)$. Let $Z$ be the component of $\Gamma(\S)$ containing $\a$. Let $\eta\in\Pi_1$ be a root such that $(\eta,\theta_Z)<0$. We have that $\eta+\theta_Z+\theta_\S$ is a root, so $k\d-\eta-\theta_\S-\theta_Z$ is a root, which is positive since its $\s$-height is $1$. It follows that $k\d\ge \theta_\S+\theta_Z$, hence $\gamma\ge \theta_\S-\be+\theta_Z$. But then $c_\a(\gamma)>0$, which is impossible. We have therefore $\S'\ne\S$. But then $\gamma=k\d-\be$ with $\be\not\in\S$, so, clearly, $\gamma\geq \th_{\S}$.
 
 It remains only to check that $\D^2_{\a,\mu}\ne\emptyset$ if and only if 
$\a\in A(\S)\setminus (\S\cup\Pi_1)$. If $\a\in A(\S)\setminus (\S\cup\Pi_1)$ then $\theta_\S\in \D(\Pia_\a)$, hence $\theta_\S\in\D^2_{\a,\mu}$. Assume now  $\gamma\in\D^2_{\a,\mu}$. If $\a\in \Pi_1\cup\S$, then $\theta_\S\not\in \D(\Pia_\a)$, hence $\gamma>\theta_\S$. This is absurd since it implies $c_\a(\gamma)>0$.
\medskip

\item{\it 2.} $\mu=k\d-\theta_\S$, $\theta_\S$ of type 2.  Then $\a\in\S$ and $B_\mu=\Pi_1$, so that the $B_\mu$-height is the $\s$-height. We shall prove that, if $\gamma\in\D(\Pia_\a)$, then $ht_\s(\gamma)\leq 1$. \par
Consider first the case $k=2$. Assume  $\gamma\in\D(\Pia_\a)$.
Notice that, if $\d-\gamma$ is a root, then it is positive, since then
$c_\a(\d-\gamma)>0$. Since $ht_\s(\d)=1$, this implies that 
$ht_\s(\gamma)\leq 1.$ Now, assume by contradiction that 
$ht_\s(\gamma)> 1$.
Since, in any case, $2\d-\gamma\in\Dap$, we obtain that $ht_\s(\gamma)=2$ and $2\d-\gamma\in\Dp_0$. In turn, this implies that $2\d-\gamma\in\D(\S)$, since $c_\a(2\d-\gamma)>0$, and $\a\in \S$. Thus, since $\theta_\S$ is of type 2, also $2\d-\gamma$ is of type 2. But this implies that $\d-\gamma$ is a root, hence that $ht_\s(\gamma)\leq 1$: a contradiction.
\par
Next, consider the case $k=1$. In case $B_n$, we have $\S=\{\a_n\}$ and $\Pi_1=\{\a_{n-1}\}$, so $\a=\a_n$ and  and $ht_\s(\gamma)=0$ for all
$\gamma\in \D(\Pia_\a)$. In case $C_n$, we have $\S=\{\a_1, \dots, \a_{n-1}\}$ and $\Pi_1=\{\a_0, \a_n\}$, so it is clear that for all $\a\in\S$,  and for all $\gamma\in \D(\Pia_\a)$, $ht_\s(\gamma)\leq 1$.
\par
\medskip
\item{\it 3.} $\mu=k\d+\be,\ \be\in \Pi_1$. In this case $\a\in \Pi_1$ and 
$B_\mu\subseteq \Pi_1$, so it is clear that, if $\Pi_1=\{\a\}$, then 
 $ht_\s(\gamma)=0$ for all $\gamma\in \D(\Pia_\a)$. If  $\Pi_1=\{\a, \be\}$, we obtain in any case that 
 $ht_\s(\gamma)\leq 1$ for all $\gamma\in \D(\Pia_\a)$.
\end{proof}

\begin{defi}\label{Walfa}
Given $\a\in \Pia$ and $\mu\in\mathcal M_\s$ such that $\mathcal I_{\a,\mu}\ne\emptyset$, we set
$$\Pia_{\a, \mu}=\Pia_\a\setminus B_\mu,$$
$$\quad \Pia_{\a, \mu}^*=
\begin{cases}
\Pia_{\a, \mu}\cup\{\th_\S\}
&\text{if  $\mu=k\d-\theta_\S, \theta_\S$ of type 1, $|\S|>1$},
\\ &\text{$\a\in A(\S)\setminus (\S\cup\Pi_1)$,}\\
\Pia_{\a, \mu} 
&\text{in all other cases;}
\end{cases}
$$
$$
\Wa_{\a, \mu}=W(\Pia_{\a, \mu}^*).
$$
\end{defi}

The main results of this Section is the following statement.  Recall that we identify a coset space with the set of minimal length coset representatives.

\begin{theorem}\label{slamu} Let $\a\in\Pia$ and $\mu\in\mathcal M_\s$ be such that $\mathcal I_{\a,\mu}\ne\emptyset$. Then the map $u\mapsto w_{\a,\mu}u$ is  a poset isomorphism between $\mathcal I_{\a,\mu}$ and  $\Wa_{\a,\mu}\backslash \Wa_{\a}$.
\end{theorem}

\begin{proof}
By Lemma \ref{Valfa}, we have only to prove that
$\Wa_{\a,\mu}\backslash \Wa_{\a}=V_{\a,\mu}.$
\par
Let $u\in V_{\a,\mu}$, $u\ne\bf 1$. To prove that $u\in\Wa_{\a,\mu}\backslash \Wa_{\a}$ we have to show that if $\beta\in \Pia_{\a,\mu}^*$,  then $u^{-1}(\beta)\in\Dap$: this is immediate from the definitions, since $ht_{B_\mu}(\be)\in \{0, 2\}$, while, for all $\gamma\in N(u)$, $ht_{B_\mu}(\gamma)=1$.
\par
Conversely, assume $u\in \Wa_{\a,\mu}\backslash \Wa_{\a}$, $u\ne\bf 1$, 
and $\gamma\in N(u)$. If, by contradiction, $ht_{B_\mu}(\gamma)=0$, then, by the biconvexity property of $N(u)$, we obtain that there exists some $\beta\in 
(\Pia_\a\setminus B_\mu)\cap N(u)$: this contradicts the definition of 
$\Wa_{\a,\mu}\backslash \Wa_{\a}$. Therefore, $ht_{B_\mu}(\gamma)>0$. 
By Lemma \ref{M2}, this implies that $ht_{B_\mu}(\gamma)=1$ in all cases except when 
$\mu=k\d-\th_\S$, with $\th_\S$ of type $1$ and $|\S|>1$. 
It remains to prove that also in this case $ht_{B_\mu}(\gamma)=1$.
First, we observe that, by Lemma \ref{supp}, $ht_{B_\S}(k\d-\th_\S)=0$: it follows that  $ht_{B_\S}(k\d)=ht_{B_\S}(\th_\S)$ and, by Lemma \ref{M2}, 
that $ht_{B_\mu}(k\d)=2$. Hence, $ht_{B_\mu}(\gamma)\leq 2$, since 
$k\d-\gamma$ is a positive root. Now, if we assume, by contradiction, that $ht_{B_\mu}(\gamma)=2$, then by Lemma \ref{M2}, we obtain that 
$\gamma$ is equal to $\th_\S$ plus a, possibly empty, sum of positive roots with null $B_\mu$-height. By the biconvexity of $N(u)$, this implies that some root in  $(\Pia_\a\setminus B_\mu)\cup\{\th_\S\}$ belongs to $N(u)$, in contradiction with the definition of 
$\Wa_{\a,\mu}\backslash \Wa_{\a}$. 
\end{proof}
\vskip5pt

\section{Intersections among $\mathcal I_{\a,\mu}$'s}\label{cap}
Our goal in this Section is the proof of the following Theorem.
\begin{theorem}\label{minimax} \par\noindent
(1). If $\a\ne \be$, then $\mathcal I_{\a,\mu}\cap \mathcal I_{\be,\mu'}\ne\emptyset$ if and only if $\mu=k\d-\theta_\S$, $\mu'=k\d-\theta_\S'$ with $\S\ne\S'$, $\a\in\S'$, $\beta\in\S$,  and $\a,\beta,\theta_\S,\theta_{\S'}$ all of type 1.
 \par\noindent
(2). If $\mathcal I_{\a,\mu}\cap \mathcal I_{\be,\mu'}\ne\emptyset$, then $\sup\{\min \mathcal I_{\a,\mu}, \min \mathcal I_{\beta,\mu'}\}$ exists. Denoting it by $w_{\a,\beta}$, we have that 
 $$\mathcal I_{\a,\mu}\cap \mathcal I_{\beta,\mu'}\cong W((\Pia_\a\cap\Pia_\be)\setminus\Pi_1)\backslash W(\Pia_\a\cap\Pia_\be),$$ the isomorphism being $u\mapsto w_{\a,\beta}u,\,u\in W((\Pia_\a\cap\Pia_\be)\setminus\Pi_1)\backslash W(\Pia_\a\cap\Pia_\be)$.\end{theorem}
\vskip5pt Statements (1), (2)  are proved in Propositions \ref{vuoto},  \ref{nonvuoto2}, respectively.
\vskip10pt

\begin{defi}\label{A-u}
Assume that $\S$ and $\S'$ are distinct components of $\Pi_0$.  We define
$$A(\S,\S')=A(\S)\cap A(\S').$$
Moreover, we set 
$$W_{\S,\S'}=W(A(\S,\S')), \qquad W_{\S,\S'}^1=W(A(\S,\S')\setminus\Pi_1),$$
and denote by $u_{\S, \S'}$ the maximal element
in $W_{\S,\S'}^1\backslash W_{\S,\S'}$. 
\end{defi}

According to Definition \ref{A-u} and Subsection \ref{posetparabolic}, 
\begin{equation}\label{Nu}
N(u_{\S, \S'})=\{\beta\in \D(A(\S, \S'))\mid ht_\s(\be)>0\}.
\end{equation}

It is clear from Definition \ref{Asigma} that $\S'\subseteq A(\S)$; in fact, we have 
the partition
\begin{equation}\label{AS-1}
A(\S)=\bigcup_{\substack{\S'|\Pi_0\\\S'\ne \S}} \S'\cup \Pi_1\cup \G(\S).
\end{equation}
From this, we obtain
\begin{equation}\label{ASS'}
A(\S,\S')=\G(\S)\cup \Pi_1\cup \G(\S')\cup\S'',
\end{equation}
where $\S''=\Pi_0\setminus(\S\cup\S')$.
In particular we obtain the partition
\begin{equation}\label{AS-2} 
A(\S)=A(\S,\S')\cup (\S'\setminus \G(\S')).
\end{equation}

\begin{rem}\label{utile}
From equation \eqref{ASS'} and Definition \ref{Asigma}, we obtain directly that
$\theta_\S$ and $\theta_{\S'}$ are orthogonal to all
the roots in $A(\S, \S')$, except the ones in $\Pi_1$.   This 
implies that $(\be, \theta_\S)\leq 0$ and 
$(\be, \theta_{\S'})\leq 0$ for all $\be\in \D(A(\S, \S'))$. Moreover, by equation \eqref{Nu}, for any $\be\in A(\S, \S')$, we have the following equivalences of conditions:
$$(\be, \theta_\S)< 0  \iff
\be\in N(u_{\S, \S'})   \iff
(\be, \theta_{\S'})<0.
$$
\end{rem}

\begin{lemma}\label{uss'} 
Let $\S$ and $\S'$ be distinct components of $\Pi_0$.  Then
$$u_{\S, \S'}\in \Wab.$$
\end{lemma}
\begin{proof}
By formula \eqref{Nu}, for all $\beta\in N(u_{\S,\S'})$, $ht(\be)>0$. Therefore, it suffices to prove that, for all $\beta\in  \D(A(\S, \S'))$, $ht_\s(\be)<2$.
Assume by contradiction that $\be\in \D(A(\S, \S'))$ and $ht_\s(\be)<2$. 
Then $ht_\s(k\d-\be)=0$, hence $k\d-\be$  belongs to some component $\S''$ of $\Pi_0$. At least one among $\S$, $\S'$, say $\S$, is not  $\S''$. Hence $(k\d-\be,\theta_{\S})=0$, which gives $(\be,\theta_{\S})=0$: this is impossible, by Remark
\ref{utile}.
\end{proof}
\vskip 5pt

\begin{rem}\label{zeta} 
If $Z$ is a connected component of  $A(\S, \S')$, then the sum of the roots in $Z$ is a root and, by the Lemma \ref {uss'}, it has $\s$-height at most 1. This implies, in particular, that  $Z$ contains at most one root of $\Pi_1$.
\end{rem}
\vskip5pt

Though we shall not need this fact, we notice that $A(\S, \S')$ is connected except in type $A_n^{(1)}$, in which case  $A(\S, \S')=\Pi_1$, with $\Pi_1$ disconnected, since $\S\ne\S'$.
\par

\begin{lemma}\label{thetasigmaminimal} 
Let $\S$ and $\S'$ be distinct components of $\Pi_0$.  If $\theta_\S$ and $\theta_{\S'}$ are both of type 1, then 
\begin{enumerate}
\item $u_{\S,\S'}(\theta_\S)=\theta_{A(\S')}=k\d-\theta_{\S'}$, and $u_{\S,\S'}$ is the element of minimal length in $\Wa$, with this property;
\item $u_{\S,\S'}^2=1$.
\end{enumerate}
\end{lemma}
\begin{proof}
\item{(1)} Set $u=u_{\S,\S'}$. Since $L(u)\subset \Pi_1$, by  Proposition \ref{theta} (3), it suffices to  show that $u(\theta_\S)=\theta_{A(\S')}=k\d-\theta_{\S'}$. This is equivalent to show that $u^{-1}(\theta_{\S'})=k\d-\theta_\S$.
 Since $\theta_{\S'}$ is of type 1, hence long,  and  $u^{-1}(\theta_{\S'})\in A(\S)$, it suffices to show that $(u^{-1}(\theta_{\S'}),\gamma)\geq 0$ for each $\gamma\in A(\S)=A(\S,\S')\cup (\S'\backslash \G(\S'))$. We know that $u=u_{0,\Pi_1}u_0$, where $u_0$ is the longest element of $W(A(\S,\S'))$ and $u_{0,\Pi_1}$ is the longest element of $A(\S,\S')\backslash \Pi_1$. Since the only roots in $A(\S,\S')$ not orthogonal to $\theta_{\S'}$ are the roots in $\Pi_1$, we see that $u^{-1}(\theta_{\S'})=u_0(\theta_{\S'})$. Thus, since $(\theta_{\S'},\gamma)\le 0$ when $\gamma\in A(\S,\S')$, we see that $(u^{-1}(\theta_{\S'}),\gamma)=(\theta_{\S'},u_0(\gamma))\ge 0$ for $\gamma\in A(\S,\S')$.
Next we deal with the case $\gamma\in\S'\setminus \Gamma(\S')$.
If $(\gamma,\theta_{\S'})=0$ then $u(\gamma)=\gamma$, hence $(u^{-1}(\theta_{\S'}),\gamma)=(\theta_{\S'},\gamma)= 0$. If instead  $(\gamma,\theta_{\S'})\ne 0$ and  $ht_\s(u(\gamma))=1$, we are done because
$(\theta^\vee_{\S'},u(\gamma))=(\theta^\vee_{\S'},\gamma)-1\geq 0$. If $ht_\s(u(\gamma))=2$
then $ht_\s(k\d-u(\gamma))=0$, so $k\d-u(\gamma)$ belongs to some component of $\Pi_0$. If this component is ${\S'}$, then $0=(k\d-u(\gamma),\theta_{\S})$ gives a contradiction, since $c_\eta(k\d-u(\gamma))\ne 0$ for all $\eta\in \S$ such that $(\eta,\theta_\S)\ne 0$.
In the other case we have $0=(k\d-u(\gamma),\theta_{\S'})$ 
hence $0=(u(\gamma),\theta_{\S'})$ and we are done.
\item{(2)} Set again $u=u_{\S,\S'}$. Since $u_0(\theta_\S)=k\d-\theta_{\S'}$ we see that, if $\a\in\Pi_1$, then  $u_0(\a)=-\a$. In fact, if $Z$ is the component of $A(\S,\S')$ containing $\a$, then, by Remark \ref{zeta}, $\a$ is the only root in $Z$ that is not orthogonal to $\theta_{\S'}$. By  \cite{IM}, it follows that $u$ is an involution 
which permutes $A(\S,\S')\setminus\Pi_1$ and maps $\a\in\Pi_1$ to $-\theta_{Z}$.   
\end{proof}
\vskip 5pt

\begin{lemma}\label{u<w}
Let $\S\ne\S'$, $\theta_\S,\theta_{\S'}$ of type 1, $\a\in A(\S)_{\ov\ell}$, $\be \in A(\S')_{\ov\ell}$, and assume that  $w\in  \mathcal I_{\a,k\d-\theta_{\S}}\cap \mathcal I_{\beta,k\d-\theta_{\S'}}$.
Then 
\begin{enumerate}
\item
$u_{\S, \S'}\leq w;$
\item
$\a\in\S'$ and $\be\in \S$;
\item
$uv_\a v_\be\leq w$,  where $v_\a$ is the element of minimal length in $W(\S')$ that maps $\a$ to $\theta_{\S'}$, and $v_\be$ is the element of minimal length in $W(\S)$ that maps $\be$ to $\theta_{\S}$. Moreover, 
$uv_\a=w_{\a, k\d-\theta_{\S}}$,  and $uv_\be=w_{\be, k\d-\theta_{\S'}}$.
\end{enumerate}
\end{lemma}

\begin{proof}
\item{(1).}
Let $u=u_{\S,\S'}$. If $u\not\le w$, then there is $\gamma\in N(u)$ such that  $\gamma\notin N(w)$. Note that $(\gamma,\theta_{\S}^\vee)=(\gamma,\theta_{\S'}^\vee)=-1$, hence $\theta_{\S}+\gamma, \theta_{\S}+\gamma+\theta_{\S'}\in \Da$. In particular we have that $k\d-\theta_\S-\gamma\in N(w)$. But then  $k\d-\theta_{\S'}+k\d-\gamma-\theta_{\S}=2k\d-\theta_\S-\gamma-\theta_{\S'}\in N(w)$, which is absurd. We have therefore $u\le w$.
\item{(2)-(3).}
From (1) we obtain that $w=uv$ with $v(\a)=\theta_{\S'}$. { Let $U=\{\beta\in N(v)\mid \theta_{\S'}-\beta\in \Dap\}$. Arguing as in the proof of Lemma \ref{minimo}, we see that $U$ is biconvex, hence there is an element $x\in W(\S')$ such that $N(x)=U$. Since $x$ satisfies the hypothesis of Proposition \ref{theta} (6), we see that there is a root $\gamma\in\S'$ such that $x=v_\gamma$, where $v_\gamma$ is the   element of  minimal length that maps $\gamma$ to $\theta_{\S'}$. We conclude that $v_\gamma\le v$.} We now show that $\ell(uv_\gamma)=\ell(u)+\ell(v_\gamma)$; for this it suffices to prove that 
$u^{-1}(\eta)=u(\eta)\in\Dap$ for $\eta\in N(v_\gamma)$. If not, then $\eta\in N(u)$, hence, by Remark \ref{utile}, $(\eta,\theta^\vee_{\S'})<0$; but $
\eta\in\S'$, hence $(\eta,\theta_{\S'})\geq 0$. We now prove that 
$L(uv_\gamma)=\Pi_1$.  We have $N(uv_\gamma)=N(u)\cup u(N(v_\gamma))$. Since $L(u)=\Pi_1$, it suffices to prove that $u(\eta)\notin\Pia$ for any $\eta\in N(v_\gamma)$.
Since $\eta\in\S'$, we have
$$(u(\eta),\theta_{\S'})=(\eta,u(\theta_{\S'}))=(\eta,k\d-\theta_{\S})=0.$$
This implies that if $u(\eta)=\xi\in \Pia$, then  $\xi\notin \Pi_1$ and,
 since $u\in W(A(\S,\S'))$, we see that, for any $\nu\in B_{\S'}$, we have $0=c_\nu(\xi)=c_\nu(u(\eta))=c_\nu(\eta)$, hence $(\eta,\theta_{\S'})=0$, against Proposition
\ref{theta} (8).
Since $uv_\gamma(\gamma)=k\d-\theta_{\S}=\theta_{A(\S')}$ and $L(uv_\gamma)\subset \Pi_1$,  we can apply
Proposition \ref{theta} (3),  to get $u v_\gamma=w_{\gamma,k\d-\theta_{\S}}$.  This implies that $w_{\gamma,k\d-\theta_{\S}}\le w$, so, by Proposition \ref{fuori}, $w\in \mathcal I_{\gamma,\mu}$, hence $\a=\gamma\in\S'$ and 
$uv_\a\le w$. Similarly, 
 $\be=\gamma\in\S$ and 
$uv_\be\le w$. Since
\begin{equation}\label{nunione}N(uv_\a v_\beta)=N(u)\cup u(N(v_\a))\cup u(N(v_\be)),\end{equation}
 we get that $uv_\a v_\be\le w$.
\end{proof}
\vskip 5pt

\begin{prop}\label{nonvuoto}
 Assume $\S\ne\S'$, $\theta_\S,\theta_{\S'}$ of type 1, $\a\in A(\S)_{\ov\ell}$, and $\be \in A(\S')_{\ov\ell}$. Then  $\mathcal I_{\a,k\d-\theta_{\S}}\cap \mathcal I_{\beta,k\d-\theta_{\S'}}\ne\emptyset$ if and only if 
$\a\in\S'$ and $\be\in \S$. In this case, 
$$\min( \mathcal I_{\a,k\d-\theta_{\S}}\cap \mathcal I_{\beta,k\d-\theta_{\S'}})= uv_\a v_\be, $$  
where $v_\a$ is the element of minimal length in $W(\S')$ that maps $\a$ to $\theta_{\S'}$, and $v_\be$ is the element of minimal length in $W(\S)$ that maps $\be$ to $\theta_{\S}$.
\end{prop}

\begin{proof}
We first prove that, if $\a\in \S'$ and $\be\in \S$, then 
$uv_\a v_\be\in   \mathcal I_{\a,k\d-\theta_{\S}}\cap \mathcal I_{\beta,k\d-\theta_{\S'}}$. Indeed, it is clear that it suffices to prove that 
$uv_\a v_\be \in \Wab$.
 As shown above 
$w_{\a,k\d-\theta_\S}=uv_\a$ and $w_{\be,k\d-\theta_{\S'}}=uv_\be$.
From \eqref{nunione}
 we deduce that $N(uv_\a v_\beta)=N(w_{\a,k\d-\theta_\S})\cup N(w_{\beta,k\d-\theta_{\S'}})$
hence $uv_\a v_\beta$ is a 
 $\s$-minuscule element. 
 The  remaining statements follow from Lemma \ref{u<w}.
\end{proof}

\begin{defi}
Let $\S\ne\S'$, $\theta_\S$ and $\theta_{\S'}$ of type 1. 
Consider $\a\in\S'_{\ov\ell}$, $\be\in \S_{\ov\ell}$ and let  $v_\a$ be the element of minimal length in $W(\S')$ that maps $\a$ to $\theta_{\S'}$ and $v_\be$ the element of minimal length in $W(\S)$ that maps $\be$ to $\theta_{\S}$. Then we set 
$$w_{\a,\be}=u_{\S, \S'}v_\a v_\be.$$
\end{defi}

\begin{prop}\label{nonvuoto2}
Let $\S\ne\S'$, $\theta_\S,\theta_{\S'}$ of type 1, 
$\a\in\S'_{\bar\ell}$ and $\be\in \S_{\bar\ell}$. Then
$$
w_{\a,\beta} =\sup\{\min \mathcal I_{\a,\mu}, \min \mathcal I_{\beta,\mu'}\}$$ 
and
$$w_{\a, \be} x\in  \mathcal I_{\a,k\d-\theta_{\S}}\cap \mathcal I_{\beta,k\d-\theta_{\S'}}$$  
if and only if 
$$x\in  W((\Pia_\a\cap\Pia_\be)\setminus\Pi_1)\backslash W(\Pia_\a\cap\Pia_\be).$$
\end{prop}

\begin{proof}{Since $N(uv_\a v_\beta)=N(w_{\a,k\d-\theta_\S})\cup N(w_{\beta,k\d-\theta_{\S'}})$, it follows that $$w_{\a,\beta} =\sup\{w_{\a,k\d-\theta_\S}, w_{\beta,k\d-\theta_{\S'}}\}.$$}
Take $x\in  W((\Pia_\a\cap\Pia_\be)\setminus\Pi_1)\backslash W(\Pia_\a\cap\Pia_\be)$.   We now show that $w_{\a,\be}x\in \mathcal I_{\a,k\d-\theta_{\S}}\cap\ \mathcal I_{\be,k\d-\theta_{\S'}}$. We may assume that $x\ne 1$, { in particular $|\S|>1$.} It suffices to see that  $w_{\a,\be}x$ is $\sigma$-minuscule. Writing  $w_{\a,\be}x=w_{\a,k\d-\theta_{\S}}v_\be x$, by the proof of  Theorem \ref{slamu}, it suffices to prove  
  that $v_\beta x\in V_{\a,k\d-\theta_{\S}}$. Since we already know that $v_\beta \in V_{\a,k\d-\theta_{\S}}$, we are left with proving that $ht_{B_\S}(v_\beta(\gamma))=1$ for each $\gamma\in
  N(x)$. We have  $(v_\beta(\gamma),\theta_{\S})=(\gamma,\beta)=0$, hence
 $ht_{B_\S}(v_\beta(\gamma))= ht_\s(v_\beta(\gamma))=
ht_\s(\gamma)\geq 1$. 
Actually,  the latter $\s$-height is $1$: if it were $2$, then $k\d-\gamma$ would belong to some component, but this is impossible since both $\a$ and $\beta$ belong to its support. 

Vice versa, assume $w_{\a,\be}x\in  \mathcal I_{\a,k\d-\theta_{\S}}\cap\ \mathcal I_{\be,k\d-\theta_{\S'}}$, with $\ell(w_{\a,\be}x)=\ell(w_{\a,\be})+\ell(x)$ { and $x\ne1$}. By Lemma \ref{orto}, we get $v_\beta x\in \Wa_\a, 
 v_\a x\in \Wa_\be$; but $v_\beta\in \Wa_\a$, hence $x\in\Wa_\a$ and similarly 
 $x\in\Wa_\be$. We are left with proving that $L(x)\subseteq\Pi_1$, so take $\gamma\in N(x)\cap \Pia$. Recall that $v_\beta x\in V_{\a,k\d-\th_{\S}}$, hence
\begin{equation}\label{p}
1=ht_{B_\S}(v_\beta(\gamma))= ht_\s(v_\beta(\gamma)).\end{equation}
 If $\gamma\notin\Pi_1$, then $v_\beta(\gamma)\in \Da_0$, so $ht_\s(v_\beta(\gamma))=0$ against \eqref{p}. Therefore $\gamma\in\Pi_1$, as desired.
 \end{proof}

\begin{prop}\label{vuoto}
Assume $\mu, \mu'\in \mathcal M_\s$, and $\a, \be\in \Pia$.
Then $\mathcal I_{\a,\mu}\cap\mathcal I_{\be,\mu'}\ne \emptyset$
if and only if either $\a=\be$ and $\mu=\mu'$, or 
$\mu=k\d-\theta_\S$, $\mu'=k\d-\theta_\S'$, with $\S\ne \S'$ and  $\theta_\S$, $\theta_{\S'}$ of type 1, $\a\in\S'_{\ov\ell}$, and $\be\in\S_{\ov\ell}$. 
\end{prop}

\begin{proof}
In Proposition \ref{nonvuoto}, we settled the cases $\mu=k\d-\theta_\S$, $\mu'=k\d-\theta_\S'$, with $\S\ne \S'$ and  $\theta_\S$, $\theta_{\S'}$ of type 1. It remains to prove that $\mathcal I_{\a,\mu}\cap\mathcal I_{\be,\mu'}= \emptyset$
in all other non trivial cases. 
\par
We suppose by contradiction that there is $w\in \mathcal I_{\a,\mu}\cap\mathcal I_{\be,\mu'}$ and treat the possibile cases one by one.
\par
\item {\it 1.} 
Let  $\a,\be\in\Pi_1$ and $\mu=k\d+\be$, $\mu'=k\d+\a$. Since $N(w_{\be,k\d+\a})\subset N(w)$ and  $w_{\be,k\d+\a}^{-1}(\a)=-k\d+\be$ we see that $\a\in N(w)$.
 If $\Pi_0=\emptyset$ then $(\a,\be)\ne 0$, so $k\d+\a+\be\in\Da$ and this implies that $k\d+\a+\be\in N(w)$. 
 This is impossible since $ht_\s(k\d+\a+\be)=4$. If  $\Pi_0\ne\emptyset$ and $\S|\Pi_0$ then $\theta_\S+\a\in\Da$, so $k\d-\theta_\S-\a\in\Dap$. 
 Since $k\d-\a=\theta_\S+k\d-\theta_\S-\a$, using  the explicit expression for $w_{\be,k\d+\a}$ given in Proposition \ref{piofap},  we see that $\theta_\S+\a=s_\a(\theta_\S)\in N(w)$. Since $\a+\be+\theta_\S\in\Da$, this implies that  $(k\d+\be)+(\a+\theta_\S)\in N(w)$ and again this gives a contradiction.
\par
\item {\it 2.} 
Let $\a,\gamma \in \Pi_1$, $\mu=k\d+\gamma$, $\mu'=k\d-\theta_\S$.  As above, we see that $\theta_\S+\gamma\in N(w_{\a,k\d+\gamma})\subset N(w)$. But then $k\d-\theta_\S+\theta_\S+\gamma=k\d+\gamma\in N(w)$ and this is impossible.
\par
\item {\it 3.} 
Let $\mu=k\d-\theta_\S$, $\mu'=k\d-\theta_{\S'}$ with $\theta_\S$ of type 2. We have clearly $\S\ne\S'$.  Assume first $\theta_\S$ complex. If $\d-\theta_\S$ is a simple root then $\Pia=\S\cup \Pi_1$ contrary to the assumption that $\S\ne\S'$. Thus $\d-\theta_\S$ is not simple. We now rely on the explicit description of $w_{\a,\mu}$ given in Lemma \ref{paola}. If $\gamma\in\Pi_1$, then $\gamma\in N(s_{\d-\theta_\S})$, hence $2\d-2\theta_\S-\gamma\in N(s_{\d-\theta_\S})\subset N(w_{\a,\mu})\subset N(w)$. But then $(2\d-\theta_{\S'})+(2\d-2\theta_\S-\gamma)=4\d-\theta_{\S'}-\gamma-\theta_{\S}\in N(w)$ and this is not possible. It remains to check the case when $\theta_\S$ is short compact. There is only a case when this occurs and $\Pi_0$ has more than one component, namely type $B^{(1)}_n$ with $\Pi_1=\{\a_{n-1}\}$. By the explicit description of $w_{\a,\mu}$ given for this case in Lemma \ref{technical}, we see that $\theta_{\S'}\in N(w_{\a,\mu})\subset N(w)$ and this gives clearly a contradiction.
\end{proof}

\section{Maximal elements and dimension formulas}
{In this Section we give a parametrization of the maximal ideals in $\Wab$ and compute their dimension.
 
As a first step in our classification of maximal ideals, we determine which $\mathcal I_{\a,\mu}$ admits maximum.}
Let $\Pi^1_1$ denote the set of roots of type 1 in $\Pi_1$.
\begin{prop}\label{mmu} \ \begin{enumerate}
\item If $\theta_\S$ is of type 1 (resp. type 2) and $\a\in{\G(\S)}_\lu\cup\Pi^1_1$ (resp. $\a\in{\S}_\lu$) then $\mathcal I_{\a,k\d-\theta_\S}$ has maximum.
\item Suppose that $\a\in\S',\be\in\S$,  and $\S\ne\S'$. If $\theta_\S, \theta_{\S'},\a,\be'$,  are all roots of type 1, then  $\mathcal I_{\a,k\d-\theta_{\S}}\cap
 \mathcal I_{\be,k\d-\theta_{\S'}}$ has maximum.
 \item If $\Pi^1_1=\Pi_1$ and $\a,\be\in\Pi_1$, then   $\mathcal I_{\a,\be+k\d}$ has maximum whenever it is nonempty.
 \end{enumerate}
\end{prop} 
\begin{proof} 
{ Recall that, by Theorem \ref{slamu}, $\mathcal I_{\a,\mu}$ is  isomorphic to $\Wa_{\a,\mu}\backslash \Wa_\a$.} The subgroup $\Wa_{\a,\mu}$ is standard parabolic for any $\a$ and $\mu$ except { when $\mu=k\d-\theta_\S$, $\theta_\S$ of type 1, $|\S|>1$, and $\a\in A(\S)\setminus (\S\cup\Pi_1)$.}
The existence of the maximum in cases (1) and (3) follows now from subsection \ref{posetparabolic}. The same applies to $\mathcal I_{\a,k\d-\theta_{\S'}}\cap
 \mathcal I_{\be,k\d-\theta_{\S}}$,
by Theorem \ref{minimax}.
\end{proof}

{We already saw in Proposition \ref{fuori} that, in many cases, the maximal elements of $\mathcal I_{\a,\mu}$ are maximal in $\Wab$. The next result deals with the remaining cases.}
\begin{prop}\label{mmassimo}
{If $\Pi_1=\Pi_1^1=\{\a,\be\}$ (with possibly $\a=\be$), $\theta_\S$ is of type 1, and $w\in \mathcal I_{\a,k\d-\theta_\S}$, then $w\le max(\mathcal I_{\a,k\d+\be})$.}
\end{prop}

\begin{proof}
{By Proposition \ref{mmu}, $\mathcal I_{\a,k\d-\theta_\S}$ has maximum. From subsection \ref{posetparabolic}, we see that its maximum is $w_{max}=w_{\a,k\d-\theta_\S}w_{0,B_\S}w_{0,\Pia_\a}$, where $w_{0,B_\S}$ is the longest element of $W(\Pia_\a\backslash B_\S)$ and $w_{0,\Pia_\a}$ is the longest element of $\Wa_\a$. Clearly there is a root $\a_\S\in\S$ such that $(\a_\S,\a)\ne 0$, and we note that this root is necessarily unique, for, otherwise, $\S\cup \{\a\}$ would contain a loop, and this is only possible in the adjoint case of type $A_n$. But in this case $\a$ is not of type 1.

 We now show that $w_{0,B_\S}w_{0,\Pia_\a}(\a_\S)=\theta_\S$. This is clear if $|\S|=1$ so we assume $|\S|>1$.
 Recall that $w_{0,\a}$ is the longest element of $W((\Pi_0)_\a)$. 
 Let $w_{B_\S}$ be the longest element of $W((\Pi_0)_\a\backslash B_\S)$. Obviously
  $N(w_{B_\S}w_{0,\a})\subset N(w_{0,B_\S}w_{0,\Pia_\a})$ and we know that $w_{B_\S}w_{0,\a}(\a_\S)=\theta_\S$. We show that $v(\a_\S)=\theta_\S$ for any $v$ such that $w_{B_\S}w_{0,\a}\le v\le w_{0,B_\S}w_{0,\Pia_\a}$. 
This is proven by induction on $\ell(v)-\ell(w_{B_\S}w_{0,\a})$. Assume that $v(\a_\S)=\theta_\S$ and $w_{B_\S}w_{0,\a}\le v<vs_\gamma\le w_{0,B_\S}w_{0,\Pia_\a}$ with $\gamma\in \Pia_\a$. 
  We need to prove that $vs_\gamma(\a_\S)=\theta_\S$.  Set $(\a_\S,\gamma^\vee)=-r$ with $r\in \ganz^+$. Then $vs_\gamma(\a_\S)=\theta_\S+rv(\gamma)$. Observe that 
  $v\in V_{\a, k\d-\theta_\S}$, so $ht_{B_\S}(v(\gamma))=1$. It follows that $ht_{B_\S}(vs_\gamma(\a_\S))=2+r$. We claim that   $ht_{B_\S}(\nu)\le 2$ for any $\nu\in \D(\Pia\backslash \{\a\})$. Indeed this is obvious if $|\Pi_1|=1$ and, in the hermitian symmetric case it follows from \eqref{exP} and the observation that, in this case, $ht_\s(\nu)\le 1$. We conclude that $r=0$ and $vs_\gamma(\a_\S)=\theta_\S$. 
  
  Having shown that $w_{0,B_\S}w_{0,\Pia_\a}(\a_\S)=\theta_\S$, we have $w_{max}(\a_\S)=w_{\a,k\d-\theta_\S}(\theta_\S)$. Now 
  $ht_\s(w_{\a,k\d-\theta_\S}(\theta_\S))=(w_{\a,k\d-\theta_\S}(\theta_\S),k\d-\theta_\S^\vee)+ht_{B_\S}(\theta_\S)=1$. This proves that $w_{max}s_{\a_\S}\in \Wab$, so $w_{max}s_{\a_\S}\in \mathcal I_{\a,k\d+\be}$ so $w_{max}\le max(\mathcal I_{\a,k\d+\be})$. 
}
\end{proof}

Proposition \ref{mmu} allows us to give the following definition:
\begin{defi}\label{defmi} If $\th_\S$ is of type 1 (resp. type 2) and 
 $\a\in{\G(\S)}_\lu$ (resp. $\a\in{\S}_\lu$), we let $MI(\a)$ be the maximum of  $\mathcal I_{\a,k\d-\theta_\S}$. If $\S\ne\S'$ and $\theta_\S, \theta_{\S'},\a\in\S',\be\in\S$ are all roots of type 1, we let $MI(\a,\be)$ be  the maximum of $\mathcal I_{\a,k\d-\theta_{\S'}}\cap
 \mathcal I_{\be,k\d-\theta_{\S}}$. If $\a,\be\in \Pi^1_1$ with $\mathcal I_{\a,k\d+\be}\ne\emptyset$, we let $MI(\a)$ be its maximum.
\end{defi}

{We are finally ready to state the main result of the paper, which gives a complete parametrization of the set of maximal abelian $\b^{\bar 0}$-stable subspaces in $\Wab$.} For notational reasons, it is convenient to fix an arbitrary total order $\prec$ on the components of $\Pi_0$. 
\begin{theorem}\label{spazioparametri}  The 
maximal $\b^{\bar 0}$-stable abelian subalgebras are parametrized by the set 
\begin{equation}\label{M}\mathcal M=\left(\bigcup_{\substack{\S|\Pi_0\\ \S\text{ of type 1}}}\G(\S)_\lu\right)\cup\left(\bigcup_{\substack{\S|\Pi_0\\ \S\text{ of type 2}}}\S_\lu\right)\cup\left( \bigcup_{\substack{\S,\S'|\Pi_0,\,\S\prec\S'\\
\S,\,\S'\text{ of type 1}}}(\S_\lu\times\S'_\lu)\right)\cup\Pi^1_1.\end{equation}
\end{theorem}

\begin{rem}\label{PS} In the adjoint case, there is just one component $\S$ in $\Pi_0$, which is the set of simple roots of $\g$. In the r.h.s of \eqref{M} the only surviving term is 
$\S_\lu$, so $\mathcal M$ is the set of long simple roots of $\g$. This parametrization has been first discovered by Panyushev and R\"ohrle \cite{PR}, \cite{Panadv}. \end{rem}

Now we begin to work in view of the proof of Theorem \ref{spazioparametri}. {We need to study the maximal elements of $\mathcal I_{\a,\mu}$. This is immediate when $\mathcal I_{\a,\mu}$ has maximum, more delicate in the other cases. We also need to determine when a maximal element of $\Wab$ occurs in different $\mathcal I_{\a,\mu}$'s. The description of the intersections among different $\mathcal I_{\a,\mu}$'s given in Section \ref{cap} is the key to solve both problems. }We start with the following 
\begin{lemma}\label{coppie} Assume $\S\ne\S'$. If  $\theta_\S, \theta_{\S'},\a\in\S$,   are all roots of type 1 and $w\in \mathcal I_{\a,k\d-\theta_{\S'}}$ is maximal, then there is $\eta\in\S'$ such that $w(\eta)=k\d-\theta_{\S}$.
\end{lemma}
\begin{proof}Write $w=w_{\a,k\d-\theta_{\S'}}x$ with $x$ maximal in $V_{\a,k\d-\theta_{\S'}}$. If $\S'=\{\theta_{\S'}\}$, then by Lemma \ref{M2} and Theorem \ref{slamu}, $x=1$, so $w(\theta_{\S'})=u_{\S,\S'}v_\a(\theta_{\S'})=k\d-\theta_{\S}$.

If $|\S'|>1$, then by Definition \ref{Walfa}, we have that $\Wa_{\a,k\d-\theta_{\S'}}\ne \{1\}$. It follows that $x$ cannot be the longest element of $W(\Pia_\a)$, hence there is a root $\gamma$ in $\Pia_\a$ such that $x(\gamma)>0$. Since $x$ is maximal, then  $ht_{B_{\S'}}(x(\gamma))\ne 1$, 
hence $ht_{B_{\S'}}(x(\gamma))\in\{0,2\}$. Next we exclude that  $ht_{B_{\S'}}(x(\gamma))=0$ for all $\gamma$. We  start with proving that 
 if  $ht_{B_{\S'}}(x(\gamma))= 0$, then $x(\gamma)$ is simple.  
Indeed, if $x(\gamma)-\be\in \Dap_\a$ with $\be\not\in B_{\S'}$, then, by convexity of $N(x)$, we have that $\be\in N(x)$, contradicting the fact that $x\in V_{\a,k\d-\theta_{\S'}}$. If, for all  roots $\gamma$ in $\Pia_\a$ such that $x(\gamma)>0$ we have that $x(\gamma)\in \Pia\backslash B_{\S'}$, then, 
arguing as in Proposition \ref{max}, we see that $N(x)$ is the  set of roots $\be$ in $\Dap_\a$ such that $ht_{B_{\S'}}(\be)>0$. Since $(\theta_{\S'},\theta_{\S'}^\vee)=2$ and $|\S'|>1$, we see that this contradicts again the fact that $x\in V_{\a,k\d-\theta_{\S'}}$. \par
Therefore  there is $\gamma$ such that $ht_{B_{\S'}}(x(\gamma))= 2$. Then, arguing as above, we see that $x(\gamma)$ is minimal among the roots $\be$ such that $ht_{B_{\S'}}(\be)= 2$. 
By Lemma \ref{M2}, we have that $x(\gamma)=\theta_{\S'}$.

{Arguing as in the proof of parts (2), (3)  of  Lemma \ref{u<w}, one checks that there is $\eta\in\S'$
 such that $v_\eta\le x$.  It follows that $w_{\a,k\d-\theta_{\S'}}v_\eta\le w$. 
 Since $w_{\a,k\d-\theta_{\S'}}v_\eta=u_{\S,\S'}v_\a v_\eta\in\mathcal I_{\eta,k\d-\theta_\S}$, 
 by Proposition \ref{fuori} we have $w\in \mathcal I_{\eta,k\d-\theta_\S}$ as desired.} 
\end{proof}

{ We are now ready to prove Theorem \ref{spazioparametri}.}
\begin{proof}[Proof of Theorem \ref{spazioparametri}]Consider the map $MI:\mathcal M\to \Wab$ defined in Definition \ref{defmi}. Let  $MAX$ be the set of maximal abelian $\b^{\bar 0}$-stable subalgebras of $\g^{\bar 1}$. 
 By Propositions \ref{fuori} and \ref{mmassimo}, it is clear that $MI(m)\in MAX$ for any $m\in\mathcal M$. We  next prove that $MI:\mathcal M\to MAX$ is bijective.
First we show that $MI(\mathcal M)=MAX$. Let $w$ be maximal. By Proposition
\ref{max} we have that $w$ is maximal in $\I_{\a,\mu}$ for some $\mu\in \mathcal M_\s$ . If $\a\in\Pi_1$ and it is of type 2, then $\mu$ is of type 2, hence $\mu=k\d-\theta_\S$ with $\theta_\S$ of type 2, but this case is ruled out by Theorem \ref{min}. We can therefore assume $\a$ of type 1. From Proposition \ref{mmassimo} we deduce $\mu=\be+k\d$ so that $\a,\be\in\Pi^1_1$. Hence  $w=MI(\a)$. If $\a\notin\Pi_1$ then, by Proposition \ref{piofap},  we have that $\mu=k\d-\theta_\S$. If $\a\in\S$ and 
$\th_\S$ is of type 1 (resp. type 2), then by Theorem  \ref{min}, we have $\a\in {\G(\S)}_\lu$ (resp. $\S_\lu$) and by Proposition \ref{mmu} we have $w=MI(\a)$. Finally assume $\a\in\S'\ne\S$. In particular, by Theorem \ref{min},  $\a$, $\theta_\S$, and $\theta_{\S'}$ are of type 1. By Lemma \ref{coppie} and Proposition \ref{mmu} (2), we see that there is $\be\in \S'$ such that $w=MI(\a,\be)$.

Finally we prove that $MI$ is injective. Set 
\begin{equation}\label{y}Y=
\bigcup_{\substack{\S|\Pi_0\\ \S\text{ of type 1}}}\G(\S)_\lu\cup\bigcup_{\substack{\S|\Pi_0\\ \S\text{ of type 2}}}\S_\lu\cup\Pi^1_1.\end{equation}
If $\a,\be\in Y$, it follows readily from Theorem \ref{minimax} that $MI(\a)= MI(\be)$ implies $\a=\be$. Theorem \ref{minimax} also implies that   $MI(\a)\ne MI(\be,\gamma)$ for $\a\in Y$ and $(\be,\gamma)\in \S_\lu\times\S'_\lu$ with $\be,\gamma, \S,\S'$ of type 1. Suppose finally that $MI( \a,\be)=MI(\gamma,\eta)$ with $\a\in \S_\lu$, $\be\in \S'_\lu$, $\gamma\in \S''_\lu$, $\eta\in\S'''_\lu$, and $\S,\S',\S'',\S'''$ all of type 1, and $\S\prec \S'$, $\S''\prec \S'''$. Set $w=MI(\a,\be)=MI(\gamma,\eta)$. We have $w\in\mathcal I_{\a,k\d-\theta_{\S'}}\cap \mathcal I_{\gamma,k\d-\theta_{\S'''}}\ne\emptyset$. Thus either $\a=\gamma$ and $\S'=\S'''$ or $\gamma\in\S'$ and $\a\in \S'''$. In the first case we have $w(\eta)=k\d-\theta_\S$ so $\be=\eta$. In the second case we have $\S=\S'''$ and $\S''=\S'$ contradicting the fact that $\S''\prec\S'''$.
\end{proof}
\vskip5pt
As a complement to  Theorem \ref{spazioparametri}, we compute the dimension of maximal abelian subspaces.  \par
Recall from \eqref{dcn} that $g_R$ denotes the dual Coxeter number of a finite irreducible root system $R$. Suppose $\S$ is a component of $\Pi_0$. To simplify notation,  we set 
$g_\S=g_{\D(\S)}$ and, if $\th_\S$ is type 1, $g_{A(\S)}=g_{\D(A(\S))}$ (note that in this case
$\D(A(\S))$ is irreducible by Remark \ref{ascon}).
 Also recall from Section \ref{ta} that $K$ is the canonical central element of $\widehat L(\g,\s)$ and  $\bf g$ is its dual Coxeter number and from Section  \ref{bs}Ê that 
we denote by $a$ the squared length of a long root in $\Dap$.
\begin{lemma}\label{t} {Let $\gamma\in\Da_{re}$. Then\begin{enumerate}
\item $(k\d+\gamma)^\vee=\frac{a}{\Vert\gamma\Vert^2} K+\gamma^\vee$. In particular, $(k\d-\theta_\S)^\vee=r_\S K-\theta_\S^\vee$.
\item
If $\theta_\S$ is of type 1, then $g_{A(\S)}={\bf g}-g_\S+2$. 
{
In particular, if $\a\in A(\S)_\lu$, then $\ell(w_{\a,k\d-\theta_\S})={\bf g}-g_\S$.
}
\item If $\theta_\S$ is of type 2 and $\a\in \S_{\lu}$, then $\ell(w_{\a,k\d-\theta_\S})={\bf g}-1$.
\item If $\a, \be \in \Pi^1_1$, and $\be\ne \a$ if $|\Pi^1_1|=2$, then $\ell(w_{\be,k\d+\a})=\mathbf g-1$.
\end{enumerate}}
\end{lemma}
\begin{proof} {We compute, using \eqref{C}:
$$(k\d+\gamma)^\vee=\frac{2k}{\Vert\gamma\Vert^2}\nu^{-1}(\d)+\gamma^\vee=
\frac{k}{a_0}\frac{\Vert \d-a_0\a_0\Vert^2}{\Vert\gamma\Vert^2}K+\gamma^\vee.$$
A { direct } inspection shows that $\frac{k\Vert \d-a_0\a_0\Vert^2}{a_0}=a$. This proves (1). 
\par
To prove 
{
the first part of
(2)
we 
}
observe that $g_{A(\S)}=ht_{
{
\Pia
}
^\vee}((k\d-\theta_\S)^\vee)+1$. The result then follows readily from (1).
{
By Proposition \ref{theta} (\ref{4}) and \eqref{wamu}, we see that if $\theta_\S$ is of type 1 and $\a\in A(\S)_\lu$ then $\ell(w_{\a,k\d-\th_\S})=g_{A(\S)}-2={\bf g}-g_\S$.
} 
For (3) recall that $w_{\a,k\d-\theta_\S}=sv_\a$, $s$ being the element of $\Wa$ described in Lemma \ref{esse} and $v_\a$ the  element of minimal length in $W(\S)$ mapping $\a$ to $\theta_\S$. It follows that $\ell(w_{\a,k\d-\theta_\S})=\ell(s)+g_\S-2$. It is therefore enough to show that $\ell(s)=\mathbf g-g_\S+1$. Start from the following formula, which  is  a variation of e.g. \cite[Exercise 3.12]{Kac}. It is
easily proved by induction on $\ell(w)$:
\begin{equation}\label{eqpas}
w^{-1}(\l)=\l-\sum_{i=1}^l(\l,\beta_j^\vee)\a_{i_j}.
\end{equation}
Here $w\in\Wa,\,\l\in\ha^*$, $s_{i_1}\cdots s_{i_l}$ is a reduced expression of $w$ and 
$\beta_j=s_{i_1}\cdots s_{i_{j-1}}(\a_{i_j})$ (so that $N(w)=\{\beta_1,\ldots,\beta_l\}$ and $l=\ell(w)$).
Applying 	\eqref{eqpas} to $w=s$ and $\l=k\d-\theta_\S$ and using Lemma \ref{esse}, we obtain that 
$s(\l)=\l-2\sum_{i=1}^lr_j\a_{i_j}$, where $r_j=\frac{\Vert \l\Vert^2}{\Vert \beta_j\Vert^2}$. In turn, recalling that $s(\mu)=\th_\S$ and applying $\tfrac{2}{(\th_\S,\th_\S)}\nu^{-1}$ to the previous equality we get 
\begin{equation*}\label{98}
\th_\S^\vee=(k\d-\theta_\S)^\vee-2\sum_{j=1}^l\a_{i_j}^\vee.
\end{equation*}
In particular, taking $ht_{\Pia^\vee}$ of both  sides, we obtain $2\ell(s)=ht_{\Pia^\vee}((k\d-\theta_\S)^\vee)-g_\S+1$. Now use part (1) (recall  that $r_\S=2$) to finish the proof.
\par
{
To prove (4), we recall that, by Proposition \ref{piofap} $w_{\be,\a+k\d}=s_\a w_{0,\a}w_0$, 
hence $N(w_{\be,\a+k\d})=\{\a\}\cup s_\a N(w_{0,\a}w_0)$. By definition, for all 
$\gamma\in N(w_{0,\a}w_0)$, we  have $(\gamma, \a^\vee) <0$, hence 
$(s_\a \gamma, \a^\vee)>0$. {
Now it is clear that  $s_\a \gamma\ne\a$, so that $s_\a \gamma-\a$ is a root. Since
$$\frac{\Vert s_\a \gamma-\a\Vert^2}{\Vert\a\Vert^2}=1+
\frac{\Vert s_\a \gamma\Vert^2}{\Vert\a\Vert^2}-(s_\a \gamma,\a^\vee).$$
and $\a$ is long, then $(s_\a \gamma, \a^\vee)=2$ and $\Vert s_\a \gamma-\a\Vert=0$
or $(s_\a \gamma, \a^\vee)=1$. The first case implies 
$s_\a \gamma=c\d+\a$  for some $c\in \real\setminus\{0\}$. This is not possible, since   $ht_\s(s_\a \gamma)=1$.} Hence  $(s_\a \gamma, \a^\vee)=1$ for all $\gamma\in  N(w_{0,\a}w_0)$. Now, formula \eqref{eqpas} with $w=w_{\be,\a+k\d}$ and $\lambda =\a+k\d$ gives 
$\be=\a+k\d-2\a-\sum_{i=2}^l (\a,\be_i^\vee)\a_{i_j}$, with $\{\be_2, \dots, \be_l\}=
s_\a  N(w_{0,\a}w_0)$ and, applying $\tfrac{2}{(\a,\a)}\nu^{-1}$, 
$$\be^\vee=K-\a^\vee-\sum\limits_{i=2}^l \a_{i_j}^\vee.$$
It follows that $l={\mathbf g}-1$, as claimed.
}}
\end{proof}
{
  If $\g$  is a simple Lie algebra,
let $g_\g$ be the dual Coxeter number of the root system of $\g$.
It is know that  ${\bf g}=g_\g$ if $\g$ is simple and 
that  ${\bf g}=g_\k$ in the adjoint case $\g=\k\oplus\k$. The following result gives our dimension formulas.}

\begin{theorem}\label{dim} If $\theta_\S$ is of type 1  and $\a\in \G(\S)_\lu$,  then
\begin{equation}\label{dim1}\dim MI(\a)=  {\bf g}-g_\S+|\Dap_\a|-|\Dp(\Pia_{\a,\mu})|.\end{equation}
 If $\a\in \Pi_1^1$, or $\a\in \S_\lu$ with $\theta_\S$ of type 2,   then
\begin{equation}\label{dim1bis}\dim MI(\a)=  {\bf g}-1+|\Dap_\a|-|\Dp(\Pia_{\a,\mu})|.\end{equation}
If $\a\in\S_\lu$, $\beta\in\S'_\lu$,  with $\S\ne\S'$ and $\theta_\S$, $\theta_{\S'}$ of type 1, then 
\begin{equation}\label{dim2}\dim MI(\a,\be)= {\bf g}-2+|\Dp(\Pia_\a\cap\Pia_\be)|-|\Dp((\Pia_\a\cap\Pia_\be)\setminus\Pi_1)|.\end{equation}
\end{theorem}
\begin{proof} By \eqref{poset},  Theorems \ref{min} and \ref{slamu}  imply 
that, for $\a\in Y$ (cf. \eqref{y})
$$\dim MI(\a)=  \ell(w_{\a,\mu})+|\Dap_\a|-|\Dp(\Pia^*_{\a,\mu})|.$$
Note that in the current setting we have that $\Pia^*_{\a,\mu}=\Pia_{\a,\mu}$. Using part (2) of the previous Lemma we obtain 
\eqref{dim1}. Likewise, if $\theta_\S$ is of type 2 and $\a\in\S_\lu$, or
 $\a\in \Pi_1^1$
 then \eqref{dim1bis} follows from  (3), (4) in Lemma \ref{t}.\par
 Finally, we have to prove \eqref{dim2}.
Theorem \ref{minimax} gives  
$$\dim MI(\a,\be)=\ell(w_{\a,\be})+|\Dp(\Pia_\a\cap\Pia_\be)|-|\Dp((\Pia_\a\cap\Pia_\be)\setminus\Pi_1)|.$$ So it remains to show that { $\ell(w_{\a,\be})=\mathbf g-2$. From Lemma \ref{u<w} (3), we know} that $w_{\a,\be}=u_{\S,\S'}v_\a v_\be$ with $\ell(w_{\a,\be})=\ell(u_{\S,\S'})+\ell(v_\a)+\ell(v_\be)$, where  $v_\a, v_\be$ are the elements of minimal length mapping $\a,\be$, respectively, to the highest root of their component. 
By Proposition \ref{theta} (4), the lengths of the latter elements are $g_\S-2$,  $g_{\S'}-2$, respectively. We know that 
$u_{\S,\S'}v_\beta$ is the element of minimal length in $W(A(\S))$  mapping $\beta$ to $k\d-\theta_\S$. Hence  $\ell(u_{\S,\S'})+\ell(v_\beta)=g_{A(\S)}-2$. Using Lemma \ref{t} (1), we have
\begin{align*}\ell(u_{\S,\S'})&=g_{A(\S)}-g_{\S'}={\bf g}-g_\S-g_{\S'}+2,\end{align*} hence \eqref{dim2} is proven.
\end{proof}

\begin{rem}\label{r72} The dimension formula in the adjoint
case is a specialization of \eqref{dim1bis} and is  due to
Suter \cite{Suter}. For  a refinement of Suter's formula, see \cite[Theorem 8.13]{CP3}.
\end{rem}
\begin{example} We illustrate our results  when $\g$ is of type $E_8$ and $\g^{\bar 0}$ of type $A_1\times E_7$. 
 We denote by $\S_1$ the component of $\Pi_0$ of type $A_1$ and by $\S_2$ that of type $E_7$. 
$$
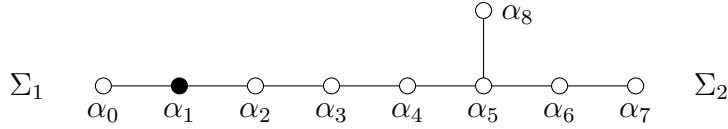
\begin{tikzpicture}[scale=1]
\pgfputat{\pgfxy(-1,0)}{\pgfbox[center,center]{$\Sigma_1$}}
\pgfputat{\pgfxy(8,0)}{\pgfbox[center,center]{$\Sigma_2$}}
\pgfputat{\pgfxy(0,-0.35)}{\pgfbox[center,center]{$\a_0$}}
\pgfputat{\pgfxy(1,-0.35)}{\pgfbox[center,center]{$\a_1$}}
\pgfputat{\pgfxy(2,-0.35)}{\pgfbox[center,center]{$\a_2$}}
\pgfputat{\pgfxy(3,-0.35)}{\pgfbox[center,center]{$\a_3$}}
\pgfputat{\pgfxy(4,-0.35)}{\pgfbox[center,center]{$\a_4$}}
\pgfputat{\pgfxy(5,-0.35)}{\pgfbox[center,center]{$\a_5$}}
\pgfputat{\pgfxy(5.45,0.95)}{\pgfbox[center,center]{$\a_8$}}
\pgfputat{\pgfxy(6,-0.35)}{\pgfbox[center,center]{$\a_6$}}
\pgfputat{\pgfxy(7,-0.35)}{\pgfbox[center,center]{$\a_7$}}
\draw{(6,0)--(7,0)};
\draw[fill=white]{(7,0) circle(3pt)};
\draw[fill=black]{(1,0) circle(3pt)};
\foreach \x in {0,1,2,3,4,5}
\draw{(\x,0)--(\x+1,0)};
\draw{(5,0)--(5,1)};
\foreach \x in {0,2,3,4,5,6}
\draw[fill=white]{(\x,0) circle(3pt)};
\draw[fill=white]{(5,1) circle(3pt)};
\end{tikzpicture}
$$
\vskip20pt
Then $A(\S_1)=\Pia\setminus\{\a_0\},\Gamma(\Sigma_1)=\emptyset,\,A(\Sigma_2)=\Pia\setminus\{\a_7\},
\Gamma(\Sigma_2)=\Sigma_2\setminus\{\a_7\}$.
Set 
$$\mu_1=-\theta_{\S_1}+\d,\quad\mu_2=-\theta_{\S_2}+\d,\quad\mu_3=\a_1+\d.$$
By Theorem \ref{min}, the poset $\mathcal I_{\a_i,\mu_j}$ is non empty if and only if 
$$(i,j)\in\{(k,1)\mid 1\leq k\leq 8\}\cup\{(k,2)\mid k=0,1,2,3,4,5,6,8\}\cup\{(1,3)\}.
$$
 By Theorem
\ref{spazioparametri}, the  maximal $\b^{\bar 0}$-stable abelian subspaces  are $MI(\a_i),\,2\leq i\leq 6$ and $i=8,\,MI(\a_i,\a_0),\,2\leq i\leq 8,\, MI(\a_1)$. More explicitly, 
\begin{align*}
&MI(\a_1)=\max \mathcal I_{\a_1,\mu_3},\\&MI(\a_i)=\max \mathcal I_{\a_i,\mu_2},\,2\leq i\leq 8, i\ne 7,\\
& MI(\a_i,\a_0)=\max(\mathcal I_{\a_0,\mu_2}\cap \mathcal I_{\a_i,\mu_1}),\,2\leq i\leq 8.\end{align*}
 The following table displays the relevant data necessary to  calculate the corresponding dimensions using the formulas provided in  Theorem \ref{dim}.
 The formula in question appears in the leftmost column.
 Recall that ${\bf g}=30$ and that $g_{E_7}=18$. 
 It is easily checked that  $B_{\mu_2}=\{\a_7\},B_{\mu_3}=\{\a_1\}$.
 $$
\begin{tabular}{l |l |l |l | l}
\eqref{dim1}&$MI(\a)$ &  type of $\Da_\a$& type of $\D(\Pia_{\a,\mu})$&$\dim MI(\a)$\\
\hline
&$MI(\a_6)$ & $A_5\times A_1$ &$A_5\times A_1$ &30-18+16-16=12\\
&$MI(\a_5)$ & $A_4\times A_1$&$A_4$ &30-18+11-10=13\\
&$MI(\a_4)$ & $A_3\times A_2\times A_1$&$A_3\times A_1\times A_1$&30-18+10-8=14\\
&$MI(\a_8)$ & $A_5\times A_2$&$A_5\times A_1$ &30-12+18-16=14\\
&$MI(\a_3)$ & $A_2\times A_4$&$A_2\times A_3$ &30-18+13-9=16\\
&$MI(\a_2)$ & $A_1\times D_5$&$A_1\times D_4$ &30-18+21-13=20\\\hline
\eqref{dim2}&$MI(\a,\be)$ &  type of $\D(\Pia_\a\cap\Pia_\be)$& type of&$\dim MI(\a)$\\
 & & & $\D((\Pia_\a\cap\Pia_\be)\setminus\Pi_1)$\\
&$MI(\a_0,\a_i),$ &$R$ &$R$ & 28=30-2\\
&$2\leq i\leq 8$ & & & \\\hline
\eqref{dim1bis}&$MI(\a)$ &  type of $\Da_\a$& type of $\D(\Pia_{\a,\mu})$&$\dim MI(\a)$\\
&$MI(\a_1)$ &$E_6$ &$E_6$ &29=30-1+36-36
\end{tabular}
$$
The symbol $R$ in the fourth to last line of the previous table means that if $\D(\Pia_\a\cap\Pia_\be)$ is of type $R$, then also  $\D((\Pia_\a\cap\Pia_\be)\setminus\Pi_1)$ is of type $R$. This happens because $\a_1\notin \Pia_{\a_0}$.
\end{example}

\begin{example}Here we illustrate another example corresponding to $\g$ of type $D_5$ and $\g^{\bar 0}$ of type $A_1\times B_3$. We denote
by $\S_1$ the component of $\Pi_0$ of type $A_1$ and by $\S_2$ that of type $B_3$. The corresponding picture is
$$
\begin{tikzpicture}[scale=1]
\pgfputat{\pgfxy(-1,0)}{\pgfbox[center,center]{$\Sigma_1$}}
\pgfputat{\pgfxy(5,0)}{\pgfbox[center,center]{$\Sigma_2$}}
\pgfputat{\pgfxy(0,-0.35)}{\pgfbox[center,center]{$\a_0$}}
\pgfputat{\pgfxy(1,-0.35)}{\pgfbox[center,center]{$\a_1$}}
\pgfputat{\pgfxy(2,-0.35)}{\pgfbox[center,center]{$\a_2$}}
\pgfputat{\pgfxy(3,-0.35)}{\pgfbox[center,center]{$\a_3$}}
\pgfputat{\pgfxy(4,-0.35)}{\pgfbox[center,center]{$\a_4$}}
\draw[fill=black]{(1,0) circle(3pt)};
\foreach \x in {1,2}
\draw{(\x,0)--(\x+1,0)};
\draw{(3,-0.05)--(4,-0.05)};
\draw{(3,0.05)--(4,0.05)};
\draw{(0,-0.05)--(1,-0.05)};
\draw{(0,0.05)--(1,0.05)};
\draw{(3.4,0.15)--(3.55,0)--(3.4,-0.15)};
\draw{(0.55,0.15)--(0.4,0)--(0.55,-0.15)};
\foreach \x in {0,2,3,4}
\draw[fill=white]{(\x,0) circle(3pt)};
\end{tikzpicture}
$$
\vskip20pt
Then $A(\Sigma_2)=\{\a_0,\a_1,\a_2\},
\Gamma(\Sigma_2)=\{\a_2\}$.
Set 
$$\mu_1=-\theta_{\S_1}+2\d=-\a_0+2\d,\quad\mu_2=-\theta_{\S_2}+2\d,\quad\mu_3=\a_1+2\d.$$
By Theorem \ref{min}, the poset $\mathcal I_{\a_i,\mu_j}$ is non empty if and only if 
$$(i,j)\in\{(0,1),(1,2),(1,3),(2,2)\}.
$$
 By Theorem
\ref{spazioparametri}, the  maximal $\b^{\bar 0}$-stable abelian subspaces  are $MI(\a_i),\,0\leq i\leq 2$. 
 Recall that ${\bf g}=8$ and that $g_{B_3}=5$. 
 It is easily checked that  $B_{\mu_1}=B_{\mu_3}=\{\a_1\},B_{\mu_2}=\{\a_3\}$.
 \par\noindent
 Since $\a_0=\theta_{\Sigma_1}$ is of type 2, we can calculate $\dim(MI(\a_0))$ using formula \eqref{dim1bis}.
 Since $\D(\Pia_{\a_0,\mu_1})=\Da_{\a_0}$, we obtain 
 $$\dim(MI(\a_0))=8-1=7.$$
Exactly the same calculation works for $\a=\a_1$, so that  $\dim(MI(\a_1))=7$. Finally $\dim(MI(\a_2))$ is computed by \eqref{dim1}.
We have that $\Pia_{\a_2}=\{\a_0,\a_4\}$, so $\Da_{\a_2}$ is of type $A_1\times A_1$ and $\Pia_{\a_2,\mu_2}=\Pia_{\a_2}$. Thus 
$$\dim(MI(\a_2))=8-5+2-2=3.$$
\end{example}

\begin{prop}\label{} In the hermitian case, if $\a\in\Pi_1$, we have
$$\dim(MI(\a))=\frac{\dim(\g^{\bar 1})}{2}.$$
\end{prop}
{
\begin{proof}  Let  $\Pi_1=\{\a,\be\}$. 
It is clear that a root of $\widehat L(\g, \s)$ has $\s$-height 1 if it is greater or equal than
exactly one among  $\a, \be$. Hence
\begin{equation}\label{dime}t^{-1}\otimes \g^{\bar 1}=\bigoplus_{\stackrel{\gamma\geq \a}{\be\notin Supp(\gamma)}}\widehat L(\g, \s)_{-\gamma}\oplus
\bigoplus_{\stackrel{\gamma\geq \be}{\a\notin Supp(\gamma)}}\widehat L(\g, \s)_{-\gamma}.\end{equation}
Since there is an { automorphism of the Dynkin diagram of $\widehat L(\g, \s)$} switching the elements of $\Pi_1$, the two summands in the r.h.s. of \eqref{dime} have both dimension 
$\dim(\g^{\bar 1})/2$. 
Set $F_{\a}=\{-\gamma\in\Da\mid \gamma\geq \be \text{ \it and } \a\not\in Supp(\gamma)\}$.
It is clear that, if $-\gamma', -\gamma''\in F_{\a}$, then $-\gamma'-\gamma''\not\in\Da$; moreover, for each 
$\eta\in \Dp_0$ { such that $-\gamma+\eta\in \Da$ we have that} $-\gamma+\eta\in F_{\a}$. It  
follows that $\displaystyle\bigoplus_{\stackrel{\gamma\geq \be}{\a\notin Supp(\gamma)}}\widehat L(\g, \s)_{-\gamma}$ is an abelian $\b^{\bar 0}$-stable subspace of $t^{-1}\otimes \g^{\bar 1}$, { hence, by Remark \ref{inL}, it corresponds to a $\b^{\bar 0}$-stable abelian subspace of $\g^{\bar 1 }$. In order to conclude the proof, we shall prove that  the element of $\Wab$ corresponding to the  latter subspace is  $MI(\a)$}. Set $z=MI(\a)$. By formula \eqref{wamu}, Theorem \ref{min}, and Lemma \ref{orto}, 
$z=s_{\be}z'$ with $z'\in W(\Pia\setminus\{\a\})$. It follows that  $N(z)\subseteq -F_{\a}$ and therefore, by the maximality of 
$MI(\a)$, that $N(z)=-F_{\a}$: this proves the claim.
\end{proof}
\bigskip

\begin{rem}
If we take $\Pi_1=\{\a_0, \be\}$, where $\a_0$ is the extra node of the extended  Dynkin diagram associated to $\g$, then { the sum $\i$ of all root subspaces corresponding to $\{\gamma\geq\be\mid \a_0\not\in Supp(\gamma)\}$} is an ideal of the Borel subalgebra of $\g$ corresponding to the simple system $\Pia\setminus\{\a_0\}$. 
Moreover, if $w$ is the element associated to this abelian ideal via Peterson's bijection { quoted in the Introduction}, then $N(w)=\{\gamma\geq\a_0\mid \beta\not\in Supp(\gamma)\}$.
Now Proposition \ref{piofap} implies that $w(\be)=\d+\a_0$, hence this ideal is included in the maximal ideal associated to $\be$ via the Panyushev{ bijection \cite{Panadv}}. By Theorem \ref {dim} and Suter { dimension} formula,  we obtain that
$\i$  is exactly this maximal ideal.
Notice that this applies to any simple root $\be$ of $\g$ that occurs with coefficient $1$ in the highest root of $\g$.
\end{rem}
}
\bigskip

\begin{rem} As recalled in the Introduction, Panyushev \cite{Pan} investigated the maximal eigenvalue of the Casimir element of $\g^{\bar 0}$ w.r.t. the Killing form of $\g$. In particular he showed that in the hermitian case $N=\frac{\dim(\g^{\bar 1})}{2}$ gives the required maximal eigenvalue.
By the previous Proposition, if $v_1,\ldots,v_N$ is any basis of $MI(\a)$, then $v_1\wedge\ldots\wedge v_N$ is an explicit eigenvector of maximal eigenvalue.
\end{rem}

\providecommand{\bysame}{\leavevmode\hbox to3em{\hrulefill}\thinspace}
\providecommand{\href}[2]{#2}

\vskip5pt
\footnotesize{

\noindent{\bf P.C.}: Dipartimento di Scienze
 Universit\`a  di Chieti-Pescara, Viale Pindaro 42,  65127 Pescara, Italy; 
{\tt cellini@sci.unich.it }

\noindent{\bf P.M.F.}: Politecnico di Milano, Polo regionale di Como, 
Via Valleggio 11, 22100 Como,
Italy; {\tt pierluigi.moseneder@polimi.it}

\noindent{\bf P.P.}: Dipartimento di Matematica, Sapienza Universit\`a di Roma, P.le A. Moro 2,
00185, Roma, Italy; {\tt papi@mat.uniroma1.it}

\noindent{\bf M.P.}: Dipartimento di Matematica, Sapienza Universit\`a di Roma, P.le A. Moro 2,
00185, Roma, Italy; {\tt pasquali@mat.uniroma1.it} }

\end{document}